\documentclass[12pt]{article}
\usepackage{cite}

\usepackage{graphicx} % Required for inserting images
\usepackage{amssymb,latexsym,amsfonts,amsmath,amsthm}
\usepackage{amssymb,latexsym,amsfonts}
\usepackage{xfrac}

\usepackage{hyperref}
\usepackage[dvipsnames]{xcolor}
\usepackage{tikz}
\usetikzlibrary{graphs}
\usetikzlibrary{graphs.standard}
\usetikzlibrary{shapes.multipart,shapes.geometric,topaths,calc, positioning, decorations.pathreplacing, calligraphy}
\usepackage{ragged2e}
\usepackage{blindtext}
\usepackage{soul}
\usepackage{array}
\usepackage{tikz,lipsum,lmodern}
\usetikzlibrary{graphs}
\usetikzlibrary{matrix}
\usetikzlibrary{graphs.standard}
\usepackage{hyperref}
\usepackage{mathalpha}
\usepackage{array}
\usepackage{float}
\usepackage[most]{tcolorbox}
\usepackage[paperwidth=8.5in,paperheight=10.5 in]{geometry}
\usepackage{caption}
\usepackage{subcaption}
\usepackage{multicol}

\usepackage{graphicx} % Required for inserting images

\title{Leaky Zero Forcing on Subgraphs of $d$-dimensional Grid Graphs with an Application to Diamond Rectangles}
\author{
Ryan Moruzzi, Jr.\thanks{Department of Mathematics, California State University, Northridge, USA, \\ Corresponding author. Email: \texttt{ryan.moruzzi@csun.edu}} \and Sagar Shah \thanks{California State University East Bay, Hayward, USA} \and Aaditeya Tripathi \thanks{California State University East Bay, Hayward, USA} 
}

\date{October 2026}

\theoremstyle{plain} 
\newtheorem{thm}{Theorem}
\newtheorem{cor}[thm]{Corollary}

\theoremstyle{definition}
\newtheorem{defn}[thm]{Definition}

\newtheorem{prop}[thm]{Proposition}
\newtheorem{lem}[thm]{Lemma}

\newtheorem{obs}[thm]{Observation}

\begin{document}

\maketitle

\begin{abstract}
We study zero forcing and $\ell$-leaky zero forcing on Diamond rectangle
graphs $DR(a,b)$, a family of induced subgraphs of 2-dimensional grid
graphs. Motivated by this family, we develop more general results for
subgraphs of $d$-dimensional grid graphs. Using $\ell$-leaky forts, we
show that for $\ell\leq 2d-1$, every nonempty $\ell$-leaky fort in a
subgraph of
$P_{n_1}\square\cdots\square P_{n_d}$
must contain a vertex of degree at most $2d-1$. This gives bounds
and, in certain cases, exact values for the $\ell$-leaky forcing number of
such subgraphs. We then apply these results to Diamond rectangle graphs,
for which we show that the zero forcing number equals the maximum nullity
and completely characterize the $\ell$-leaky forcing number for all
$\ell\geq 1$.
\end{abstract}



\section{Introduction}
Zero forcing on a graph is an iterative graph coloring process where one begins with a subset of vertices colored blue and repeatedly applies a color change rule with the goal of eventually coloring the entire graph blue. The standard color change rule is as follows: if a blue vertex has exactly one non–blue neighbor, then that neighbor is \emph{forced} to become blue. This forcing may occur over multiple rounds until no further forces are possible. If an initial blue set of vertices is able to color all other vertices of $G$, it is called a \emph{zero forcing set}, with the minimum cardinality of such a set being called the \emph{zero forcing number} of $G$, denoted $Z(G)$.

Although zero forcing can be studied purely from a graph-theoretic perspective, it is closely connected to linear algebra through its relationship with the maximum nullity of a graph. This connection was first established by the AIM (American Institute of Mathematics) Minimum Rank–Special Graphs Work Group \cite{AIMMINIMUMRANKWorkgroup}, who introduced zero forcing in the context of the minimum rank and maximum nullity. They showed that the zero forcing number provides a graph-theoretic upper bound on the maximum nullity, and since then, the topic of zero forcing has flourished into a rich area of research in its own right. 

Zero forcing can also be viewed as a model for the spread of information in a network, and results in this area have found applications in contexts such as quantum control \cite{QuantumControl} and search algorithms \cite{SearchTrees}. In such applications, one may be concerned with the presence of faulty vertices that obstruct propagation. Therefore, the notion of a ``leaked'' or ``faulty'' vertex was introduced in \cite{KenterOGpaper}, leading to a variation of zero forcing known as \emph{leaky zero forcing} (or simply \emph{leaky forcing}) which follows the same color change rule as zero forcing, except that leaked vertices are not allowed to perform forces. The main question in this setting is: what is the minimum number of vertices required to force the graph despite the presence of $\ell$ leaked vertices, regardless of where those leaks occur? The minimum number needed is called the \emph{$\ell$-leaky forcing number} of a graph $G$, and is denoted $Z_{(\ell)}(G)$. As in standard zero forcing, \textit{forts} are key objects of interest. A \emph{fort} is a subset of vertices whose complement cannot force into it under the standard color change rule. In the leaky setting, understanding the structure of the analogous leaky forts is an essential tool for determining $Z_{(\ell)}(G)$. This concept of leaky forcing has attracted much recent attention; see, for instance, \cite{KenterOGpaper,Reslience,GeneralLeaky,herrman2024leakyforcingresiliencecartesian,QdLeaky}, as well as related variations in \cite{elias2024leakypositivesemidefiniteforcing}.

In this work, we study the zero forcing and leaky forcing numbers of
subgraphs of $d$-dimensional grid graphs in general, then focus our study
on a particular family of induced subgraphs of the 2-dimensional grid
$P_{n_1}\square P_{n_2}$. Our interest in this setting was motivated by
the family of planar graphs traditionally known as Aztec rectangle graphs,
which can be viewed as induced subgraphs of 2-dimensional grid graphs. These graphs arise as a natural generalization of the Aztec diamond graph, originally studied from a combinatorial perspective through the enumeration of domino tilings of the Aztec diamond, the dual of the Aztec diamond graph \cite{FirstMentionAD,AztecDiamondGraphs-Benkart}. More recently, extensions to tilings of double Aztec rectangles have also been studied \cite{DoubleAztecRect}.

The paper is organized as follows. In Section \ref{sec:background}, we provide definitions and known results that are used in our work. In Section 3, we establish structural results describing leaky forts in subgraphs of $d$-dimensional grid graphs, which we then use to characterize the leaky forcing number of such subgraphs. In Section \ref{sec:ApplicationHopiRec}, we further motivate our focus on Aztec rectangle graphs, establish this family as an induced subgraph of $2$-dimensional grid graphs, and obtain an upper bound on the minimum rank. In Section \ref{sec:ZFonHopiRectanlges}, we apply our general results to prove equality of the zero forcing number and maximum nullity, and we fully characterize the leaky forcing numbers for Aztec rectangle graphs. We conclude in Section \ref{sec:FutureWork} with directions for future work.

\section{Preliminaries}\label{sec:background}

All graphs we consider are finite, simple, and undirected. For a graph $G$, we write $V(G)$ and $E(G)$ for its \emph{vertex} and \emph{edge sets}, respectively, and $|V(G)|$ for the number of vertices of $G$. For $v\in V(G)$, let $N_G(v)$ denote the open neighborhood of $v$ in $G$, that is, the set of vertices adjacent to $v$, and let $\deg_G(v)=|N_G(v)|$ denote its degree.  The \emph{maximum degree} of a graph $G$ is the largest possible degree among all vertices of $G$. A graph $H$ is a \emph{subgraph of $G$} if $V(H)\subseteq V(G)$ and $E(H)\subseteq E(G)$. If $A\subseteq V(G)$, then $G[A]$ denotes the induced subgraph of $G$ on $A$, where two vertices in $A$ are adjacent in $G[A]$ if and only if they are adjacent in $G$. An \emph{edge covering} by subgraphs of $G$ is a collection of subgraphs $H_1,\dots,H_t$ with $E(G)=\bigcup_{i=1}^t E(H_i)$. We write $P_n$ for the path on $n$ vertices and $C_n$ for the cycle on $n$ vertices. The \emph{$d$-dimensional grid graph} is the Cartesian product
$P_{n_1}\square P_{n_2}\square \cdots \square P_{n_d}$. We identify $V(P_n)$ with ${0,\ldots,n-1}$, so that the vertices of $P_{n_1}\square\cdots\square P_{n_d}$ are identified with
\[\{0,1,...,n_1-1\}\times \cdots \times \{0,1,...,n_d-1\}.\]
Two vertices $(x_1,\ldots,x_d), (y_1,\ldots,y_d) \in P_{n_1}\square \cdots \square P_{n_d}$ are adjacent if and only if, for some $k\in\{1,\ldots,d\}$, $x_k$ is adjacent to $y_k$ in $P_{n_k}$ and $x_j=y_j$ for every $j\neq k$.

Let $\mathcal{S}(G)$ be the set of real symmetric matrices whose graph is $G$, that is 
$$\mathcal{S}(G) = \{A= (a_{i,j})\in \mathbb{R}^{n\times n} : A = A^T,\ a_{i,j}\ne 0 \text{ iff } ij\in E(G)\ \text{and} \ i\ne j\}.$$

The minimum rank and maximum nullity of $G$ are given by
\[
\operatorname{mr}(G)=\min\{\text{rank}(A) : A\in \mathcal{S}(G)\},\qquad
\operatorname{M}(G)=\max\{\text{null}(A) : A\in \mathcal{S}(G)\}  .
\]

Zero forcing is an iterative coloring process on $G$ where, starting with an initial blue set $B\subseteq V(G)$, if a blue vertex $u$ has exactly one white neighbor $v$, then $u$ \emph{forces} $v$ (written $u\to v$), turning $v$ blue. Starting with a subset of vertices $B$ colored blue, a
\emph{set of forces given $B$} is a set $\mathcal{F}$ of
forces that can be carried out in some order starting from $B$.
A \emph{chronological ordering of $\mathcal{F}$ given $B$}
is an ordering
\[
u_1\to v_1,\ldots,u_k\to v_k
\]
such that $u_i\to v_i$ is a valid force when the blue set is
$B\cup\{v_1,\ldots,v_{i-1}\}$.
A \emph{forcing process} $F$ of $B$ in $G$ is a set of
forces given $B$ whose chronological ordering turns the
entire graph blue.
We denote by $\mathcal{F}(B)$ the set of all forces that
can occur in some chronological ordering starting from $B$. That is, $u\to v\in\mathcal{F}(B)$ if there exists a forcing process $F$ of $B$ in $G$ that contains $u\to v$. The forces in a forcing process may also be organized into \emph{rounds}.
Rather than applying forces one at a time in a chronological ordering,
each round consists of forces that are valid from the current blue set.
Vertices that become blue during a round may perform forces beginning
in the next round.

The initial set $B$ of blue vertices is a \emph{zero forcing set} if there is a forcing process $F$ that turns the entire graph blue. The minimum size of a zero forcing set is the \emph{zero forcing number of $G$}, denoted $Z(G)$.

In 2006, at an AIM workshop, the Special Graphs Work Group determined the following connection between the maximum nullity $M(G)$ and the zero forcing number $Z(G)$ of a graph. This work was published in 2008 and helped establish zero forcing as an area of study in its own right in the context of maximum nullity and minimum rank.   

\begin{prop}[Proposition 2.4, \cite{AIMMINIMUMRANKWorkgroup}]\label{thm:NullityInequality}
    For any graph $G$, $M(G)\le Z(G)$.
\end{prop}

In 2019, the notion of a leaky vertex to hinder the zero forcing process in graphs, akin to faulty sensors while observing a network, was introduced \cite{KenterOGpaper}. Leaky zero forcing follows the same color change rule as zero forcing, with the added notion of allowing up to $\ell$ vertices to be designated as leaks, meaning they are vertices that are unable to perform a force. The minimum size of an initial blue set that is able to force the entire graph blue for every placement of $\ell$ leaks is called the \emph{$\ell$-leaky forcing number}, denoted $Z_{(\ell)}(G)$. We collect here several results about leaky forcing that will be useful throughout this paper. First, Lemma~\ref{lem:moreleaksinq} gives us that leaky forcing numbers are monotonic in terms of the number of leaks. Lemma~\ref{lem:lowdegree} gives that vertices of particular degree must always belong to particular leaky forcing sets. Finally, Theorem \ref{prop:twoways} gives a way to show that an initial blue set is a $1$-leaky forcing set by showing that every vertex can be forced in two distinct ways.

\begin{lem}[Lemma 2.1, \cite{KenterOGpaper}]
\label{lem:moreleaksinq}
For any graph $G$ on $n$ vertices,
\[  
Z_{(0)}(G) \leq Z_{(1)}(G) \leq Z_{(2)}(G) \leq \cdots \leq Z_{(n)}(G).  
\]
\end{lem}

\begin{lem}[Lemma 2.4, \cite{KenterOGpaper}] \label{lem:lowdegree}
Let $G$ be a graph. For each $\ell\geq 0$, let $S_\ell=\{v\in V(G):\deg_G(v)\leq \ell\}$. Then every $\ell$-leaky forcing set of $G$ contains $S_{\ell}$.
\end{lem}

\begin{thm}[Theorem 2.3, \cite{Reslience}]\label{prop:twoways}
A set $B$ is a $1$-leaky forcing set if and only if for every $v \in V(G)\setminus B$, there exist distinct forces $x \to v$ and $y \to v$ with $x \neq y$ in the set of all forces originating from $B$.
\end{thm}

In our work on subgraphs of the $d$-dimensional grid in Section \ref{sec:FortsInducedSubgraphs}, we utilize the notion of \emph{leaky forts} to obtain bounds on the $\ell$-leaky forcing number. Intuitively, a leaky fort is a vertex set $S$ that is resistant to forcing from the ``outside''. Specifically, if all vertices in $V(G)\setminus S$ are blue, then at most $\ell$ vertices outside $S$ can force into $S$, so $\ell$ leaks can prevent any force into $S$.

\begin{defn}\label{def:lfort}
An \emph{$\ell$-leaky fort} in a graph $G$ is a subset $S\subseteq V(G)$ such that
\[
\left|\left\{x\in V(G)\setminus S:\ \left|N_G(x)\cap S\right|=1\right\}\right|\le \ell.
\]
\end{defn}

An illustration of a fort and an example of a 3-leaky fort of a subgraph of the $2$-dimensional grid $P_{5} \square P_{5}$ can be seen in Figure \ref{fig:LeakyForts}. The next result establishes the connection between leaky forts and the $\ell$-leaky forcing number of a graph, which will play a crucial role in determining $\ell$-leaky forcing numbers in Section~\ref{sec:FortsInducedSubgraphs}.

\begin{prop}[Proposition 2.3, \cite{KenterOGpaper}]\label{prop:LeakyForcingSetIntersectsForts}
For any graph $G$, a subset of vertices $S$ of $G$ is an $\ell$-leaky forcing set if and only if $S$ intersects every $\ell$-leaky fort of $G$.
\end{prop}

\begin{figure}[h!]
    \centering

\tikzset{every picture/.style={line width=0.75pt}} %set default line width to 0.75pt        

\begin{tikzpicture}[x=0.75pt,y=0.75pt,yscale=-1,xscale=1]
%uncomment if require: \path (0,3166); %set diagram left start at 0, and has height of 3166

%Straight Lines [id:da03672671764989199] 
\draw    (191.5,1930) -- (244.58,1908.42) ;
%Straight Lines [id:da8050226245565235] 
\draw    (186.5,1919) -- (244.58,1908.42) ;
%Shape: Ellipse [id:dp7025065750291334] 
\draw   (114,1963.33) .. controls (114,1923.02) and (163.02,1890.33) .. (223.5,1890.33) .. controls (283.98,1890.33) and (333,1923.02) .. (333,1963.33) .. controls (333,2003.65) and (283.98,2036.33) .. (223.5,2036.33) .. controls (163.02,2036.33) and (114,2003.65) .. (114,1963.33) -- cycle ;
%Curve Lines [id:da807344510283815] 
\draw    (222,1985.33) .. controls (247,2006.33) and (201,2012.33) .. (226,2036.33) ;
%Curve Lines [id:da43608542481098445] 
\draw    (219,1937.33) .. controls (244,1958.33) and (197,1961.33) .. (222,1985.33) ;
%Curve Lines [id:da959365123946525] 
\draw    (219.67,1890.67) .. controls (244.67,1911.67) and (194,1913.33) .. (219,1937.33) ;
%Shape: Ellipse [id:dp6567127404232125] 
\draw   (147.79,1955.63) .. controls (154.98,1925.5) and (174.5,1903.63) .. (191.39,1906.79) .. controls (208.28,1909.95) and (216.15,1936.94) .. (208.96,1967.08) .. controls (201.77,1997.21) and (182.25,2019.07) .. (165.35,2015.91) .. controls (148.46,2012.75) and (140.6,1985.76) .. (147.79,1955.63) -- cycle ;
%Shape: Circle [id:dp790725579330702] 
\draw  [fill={rgb, 255:red, 255; green, 255; blue, 255 }  ,fill opacity=1 ] (241.33,1908.42) .. controls (241.33,1906.62) and (242.79,1905.17) .. (244.58,1905.17) .. controls (246.38,1905.17) and (247.83,1906.62) .. (247.83,1908.42) .. controls (247.83,1910.21) and (246.38,1911.67) .. (244.58,1911.67) .. controls (242.79,1911.67) and (241.33,1910.21) .. (241.33,1908.42) -- cycle ;
%Straight Lines [id:da38069210789191854] 
\draw    (310.72,1945.18) -- (294.02,1929.69) ;
%Straight Lines [id:da672090615026622] 
\draw    (290.35,1951.8) -- (294.02,1929.69) ;
%Shape: Circle [id:dp5342917443235455] 
\draw  [fill={rgb, 255:red, 255; green, 255; blue, 255 }  ,fill opacity=1 ] (293.68,1932.92) .. controls (291.9,1932.74) and (290.6,1931.14) .. (290.79,1929.35) .. controls (290.98,1927.57) and (292.58,1926.27) .. (294.36,1926.46) .. controls (296.15,1926.64) and (297.44,1928.24) .. (297.25,1930.03) .. controls (297.07,1931.81) and (295.47,1933.11) .. (293.68,1932.92) -- cycle ;
%Straight Lines [id:da022957922997620783] 
\draw    (175.33,1944.33) -- (254.92,1943.92) ;
%Shape: Circle [id:dp49472797969903337] 
\draw  [fill={rgb, 255:red, 255; green, 255; blue, 255 }  ,fill opacity=1 ] (251.67,1943.92) .. controls (251.67,1942.12) and (253.12,1940.67) .. (254.92,1940.67) .. controls (256.71,1940.67) and (258.17,1942.12) .. (258.17,1943.92) .. controls (258.17,1945.71) and (256.71,1947.17) .. (254.92,1947.17) .. controls (253.12,1947.17) and (251.67,1945.71) .. (251.67,1943.92) -- cycle ;
%Straight Lines [id:da4584372466201041] 
\draw    (169.5,1959.33) -- (254.92,1959.58) ;
%Shape: Circle [id:dp9707184803714094] 
\draw  [fill={rgb, 255:red, 255; green, 255; blue, 255 }  ,fill opacity=1 ] (251.67,1959.58) .. controls (251.67,1957.79) and (253.12,1956.33) .. (254.92,1956.33) .. controls (256.71,1956.33) and (258.17,1957.79) .. (258.17,1959.58) .. controls (258.17,1961.38) and (256.71,1962.83) .. (254.92,1962.83) .. controls (253.12,1962.83) and (251.67,1961.38) .. (251.67,1959.58) -- cycle ;
%Straight Lines [id:da22885951151879702] 
\draw    (177.5,2000) -- (255.25,2000.92) ;
%Shape: Circle [id:dp25405234279541] 
\draw  [fill={rgb, 255:red, 255; green, 255; blue, 255 }  ,fill opacity=1 ] (252,2000.92) .. controls (252,1999.12) and (253.46,1997.67) .. (255.25,1997.67) .. controls (257.04,1997.67) and (258.5,1999.12) .. (258.5,2000.92) .. controls (258.5,2002.71) and (257.04,2004.17) .. (255.25,2004.17) .. controls (253.46,2004.17) and (252,2002.71) .. (252,2000.92) -- cycle ;
\draw   (265.5,1943.33) .. controls (270,1958.61) and (274.5,1967.78) .. (279,1970.83) .. controls (274.5,1973.89) and (270,1983.06) .. (265.5,1998.33) ;
%Shape: Ellipse [id:dp06408843986926971] 
\draw  [fill={rgb, 255:red, 208; green, 2; blue, 27 }  ,fill opacity=1 ] (431.77,1911.93) .. controls (431.77,1909.37) and (433.86,1907.3) .. (436.44,1907.3) .. controls (439.02,1907.3) and (441.11,1909.37) .. (441.11,1911.93) .. controls (441.11,1914.48) and (439.02,1916.55) .. (436.44,1916.55) .. controls (433.86,1916.55) and (431.77,1914.48) .. (431.77,1911.93) -- cycle ;
%Shape: Ellipse [id:dp7199389784019377] 
\draw   (469.32,1911.93) .. controls (469.32,1909.37) and (471.41,1907.3) .. (473.99,1907.3) .. controls (476.57,1907.3) and (478.66,1909.37) .. (478.66,1911.93) .. controls (478.66,1914.48) and (476.57,1916.55) .. (473.99,1916.55) .. controls (471.41,1916.55) and (469.32,1914.48) .. (469.32,1911.93) -- cycle ;
%Straight Lines [id:da9504825391335379] 
\draw    (441.11,1911.93) -- (469.32,1911.93) ;
%Straight Lines [id:da10630953381148334] 
\draw    (436.44,1916.55) -- (436.44,1944.59) ;
%Straight Lines [id:da7337796348196598] 
\draw    (473.99,1916.55) -- (473.8,1944.59) ;
%Shape: Ellipse [id:dp6279651307054729] 
\draw  [fill={rgb, 255:red, 208; green, 2; blue, 27 }  ,fill opacity=1 ] (431.77,1948.93) .. controls (431.77,1946.37) and (433.86,1944.3) .. (436.44,1944.3) .. controls (439.02,1944.3) and (441.11,1946.37) .. (441.11,1948.93) .. controls (441.11,1951.48) and (439.02,1953.55) .. (436.44,1953.55) .. controls (433.86,1953.55) and (431.77,1951.48) .. (431.77,1948.93) -- cycle ;
%Shape: Ellipse [id:dp4267699150759492] 
\draw  [fill={rgb, 255:red, 208; green, 2; blue, 27 }  ,fill opacity=1 ] (469.32,1948.93) .. controls (469.32,1946.37) and (471.41,1944.3) .. (473.99,1944.3) .. controls (476.57,1944.3) and (478.66,1946.37) .. (478.66,1948.93) .. controls (478.66,1951.48) and (476.57,1953.55) .. (473.99,1953.55) .. controls (471.41,1953.55) and (469.32,1951.48) .. (469.32,1948.93) -- cycle ;
%Straight Lines [id:da1318818619044183] 
\draw    (441.11,1948.93) -- (469.32,1948.93) ;
%Shape: Ellipse [id:dp33128612244237954] 
\draw  [fill={rgb, 255:red, 208; green, 2; blue, 27 }  ,fill opacity=1 ] (507.32,1948.93) .. controls (507.32,1946.37) and (509.41,1944.3) .. (511.99,1944.3) .. controls (514.57,1944.3) and (516.66,1946.37) .. (516.66,1948.93) .. controls (516.66,1951.48) and (514.57,1953.55) .. (511.99,1953.55) .. controls (509.41,1953.55) and (507.32,1951.48) .. (507.32,1948.93) -- cycle ;
%Straight Lines [id:da22471099020408591] 
\draw    (479.11,1948.93) -- (507.32,1948.93) ;
%Straight Lines [id:da9824690266934843] 
\draw    (474.44,1953.55) -- (474.44,1981.59) ;
%Straight Lines [id:da31337034898968263] 
\draw    (511.99,1953.55) -- (511.8,1981.59) ;
%Shape: Ellipse [id:dp7956669259457142] 
\draw  [fill={rgb, 255:red, 208; green, 2; blue, 27 }  ,fill opacity=1 ] (469.77,1985.93) .. controls (469.77,1983.37) and (471.86,1981.3) .. (474.44,1981.3) .. controls (477.02,1981.3) and (479.11,1983.37) .. (479.11,1985.93) .. controls (479.11,1988.48) and (477.02,1990.55) .. (474.44,1990.55) .. controls (471.86,1990.55) and (469.77,1988.48) .. (469.77,1985.93) -- cycle ;
%Shape: Ellipse [id:dp1514363314856314] 
\draw   (507.32,1985.93) .. controls (507.32,1983.37) and (509.41,1981.3) .. (511.99,1981.3) .. controls (514.57,1981.3) and (516.66,1983.37) .. (516.66,1985.93) .. controls (516.66,1988.48) and (514.57,1990.55) .. (511.99,1990.55) .. controls (509.41,1990.55) and (507.32,1988.48) .. (507.32,1985.93) -- cycle ;
%Straight Lines [id:da254277077751062] 
\draw    (479.11,1985.93) -- (507.32,1985.93) ;
%Shape: Ellipse [id:dp0650130825317553] 
\draw  [fill={rgb, 255:red, 208; green, 2; blue, 27 }  ,fill opacity=1 ] (512.55,1916.83) .. controls (509.99,1916.79) and (507.95,1914.67) .. (507.99,1912.09) .. controls (508.03,1909.51) and (510.13,1907.45) .. (512.69,1907.49) .. controls (515.24,1907.53) and (517.28,1909.65) .. (517.24,1912.23) .. controls (517.2,1914.81) and (515.1,1916.87) .. (512.55,1916.83) -- cycle ;
%Straight Lines [id:da22240119301217076] 
\draw    (512.11,1945.04) -- (512.55,1916.83) ;
%Straight Lines [id:da9884356379104333] 
\draw    (516.67,1949.78) -- (544.7,1950.21) ;
%Straight Lines [id:da6267180643847299] 
\draw    (517.24,1912.23) -- (545.28,1912.85) ;
%Shape: Ellipse [id:dp2864601416944794] 
\draw  [fill={rgb, 255:red, 208; green, 2; blue, 27 }  ,fill opacity=1 ] (548.96,1954.94) .. controls (546.41,1954.91) and (544.37,1952.78) .. (544.41,1950.2) .. controls (544.45,1947.62) and (546.55,1945.57) .. (549.11,1945.6) .. controls (551.66,1945.64) and (553.7,1947.77) .. (553.66,1950.35) .. controls (553.62,1952.93) and (551.52,1954.98) .. (548.96,1954.94) -- cycle ;
%Shape: Ellipse [id:dp09912534382882021] 
\draw  [fill={rgb, 255:red, 208; green, 2; blue, 27 }  ,fill opacity=1 ] (549.54,1917.4) .. controls (546.99,1917.36) and (544.95,1915.23) .. (544.99,1912.66) .. controls (545.03,1910.08) and (547.13,1908.02) .. (549.68,1908.06) .. controls (552.24,1908.1) and (554.28,1910.22) .. (554.24,1912.8) .. controls (554.2,1915.38) and (552.1,1917.44) .. (549.54,1917.4) -- cycle ;
%Straight Lines [id:da09920137469346024] 
\draw    (549.11,1945.6) -- (549.54,1917.4) ;
%Shape: Ellipse [id:dp8018486298101105] 
\draw   (475.12,1879.25) .. controls (472.56,1879.21) and (470.52,1877.09) .. (470.56,1874.51) .. controls (470.6,1871.93) and (472.71,1869.87) .. (475.26,1869.91) .. controls (477.82,1869.95) and (479.86,1872.07) .. (479.82,1874.65) .. controls (479.78,1877.23) and (477.67,1879.29) .. (475.12,1879.25) -- cycle ;
%Straight Lines [id:da14694148354867154] 
\draw    (474.69,1907.46) -- (475.12,1879.25) ;
%Straight Lines [id:da1754733928315778] 
\draw    (479.24,1912.2) -- (507.27,1912.63) ;
%Straight Lines [id:da9885781794280438] 
\draw    (479.82,1874.65) -- (507.56,1875.08) ;
%Shape: Ellipse [id:dp1488686426008936] 
\draw   (512.11,1879.82) .. controls (509.56,1879.78) and (507.52,1877.66) .. (507.56,1875.08) .. controls (507.6,1872.5) and (509.7,1870.44) .. (512.26,1870.48) .. controls (514.81,1870.52) and (516.85,1872.64) .. (516.81,1875.22) .. controls (516.77,1877.8) and (514.67,1879.86) .. (512.11,1879.82) -- cycle ;
%Straight Lines [id:da9814666915224396] 
\draw    (511.68,1908.03) -- (512.11,1879.82) ;
%Shape: Ellipse [id:dp9837949890713518] 
\draw  [fill={rgb, 255:red, 208; green, 2; blue, 27 }  ,fill opacity=1 ] (432.32,1986.93) .. controls (432.32,1984.37) and (434.41,1982.3) .. (436.99,1982.3) .. controls (439.57,1982.3) and (441.66,1984.37) .. (441.66,1986.93) .. controls (441.66,1989.48) and (439.57,1991.55) .. (436.99,1991.55) .. controls (434.41,1991.55) and (432.32,1989.48) .. (432.32,1986.93) -- cycle ;
%Straight Lines [id:da4973172305609055] 
\draw    (436.44,1953.55) -- (436.44,1981.59) ;
%Straight Lines [id:da9660937730694599] 
\draw    (441.11,1985.93) -- (469.32,1985.93) ;
%Shape: Ellipse [id:dp6365886347677558] 
\draw   (395.32,1987.93) .. controls (395.32,1985.37) and (397.41,1983.3) .. (399.99,1983.3) .. controls (402.57,1983.3) and (404.66,1985.37) .. (404.66,1987.93) .. controls (404.66,1990.48) and (402.57,1992.55) .. (399.99,1992.55) .. controls (397.41,1992.55) and (395.32,1990.48) .. (395.32,1987.93) -- cycle ;
%Shape: Ellipse [id:dp8689308269956346] 
\draw  [fill={rgb, 255:red, 208; green, 2; blue, 27 }  ,fill opacity=1 ] (394.77,1949.93) .. controls (394.77,1947.37) and (396.86,1945.3) .. (399.44,1945.3) .. controls (402.02,1945.3) and (404.11,1947.37) .. (404.11,1949.93) .. controls (404.11,1952.48) and (402.02,1954.55) .. (399.44,1954.55) .. controls (396.86,1954.55) and (394.77,1952.48) .. (394.77,1949.93) -- cycle ;
%Straight Lines [id:da9998740154705188] 
\draw    (404.11,1949.93) -- (432.32,1949.93) ;
%Straight Lines [id:da04660606172964621] 
\draw    (399.44,1954.55) -- (399.44,1982.59) ;
%Straight Lines [id:da7925703094530913] 
\draw    (404.11,1986.93) -- (432.32,1986.93) ;
%Shape: Ellipse [id:dp6951022351654654] 
\draw   (432.32,2024.93) .. controls (432.32,2022.37) and (434.41,2020.3) .. (436.99,2020.3) .. controls (439.57,2020.3) and (441.66,2022.37) .. (441.66,2024.93) .. controls (441.66,2027.48) and (439.57,2029.55) .. (436.99,2029.55) .. controls (434.41,2029.55) and (432.32,2027.48) .. (432.32,2024.93) -- cycle ;
%Straight Lines [id:da6316599807316118] 
\draw    (436.44,1991.55) -- (436.44,2019.59) ;
%Straight Lines [id:da2598179139328106] 
\draw    (473.99,1991.55) -- (473.8,2019.59) ;
%Shape: Ellipse [id:dp6535428932478436] 
\draw   (469.32,2023.93) .. controls (469.32,2021.37) and (471.41,2019.3) .. (473.99,2019.3) .. controls (476.57,2019.3) and (478.66,2021.37) .. (478.66,2023.93) .. controls (478.66,2026.48) and (476.57,2028.55) .. (473.99,2028.55) .. controls (471.41,2028.55) and (469.32,2026.48) .. (469.32,2023.93) -- cycle ;
%Straight Lines [id:da5763608838567678] 
\draw    (441.11,2023.93) -- (469.32,2023.93) ;

% Text Node
\draw (212.67,1865) node [anchor=north west][inner sep=0.75pt]    {$G$};
% Text Node
\draw (132.67,1939.67) node [anchor=north west][inner sep=0.75pt]    {$S$};
% Text Node
\draw (181.74,1970.38) node [anchor=north west][inner sep=0.75pt]    {$\vdots $};
% Text Node
\draw (246.24,1966.21) node [anchor=north west][inner sep=0.75pt]    {$\vdots $};

% Text Node
\draw (282.24,1964.21) node [anchor=north west][inner sep=0.75pt]    {$\ell$};
% Text Node
% \draw (279.67,1959.17) node [anchor=north west][inner sep=0.75pt]    {$\ell $};
% % Text Node
% \draw (135.83,2042.5) node [anchor=north west][inner sep=0.75pt]  [font=\small]  {$\mathrm{( a) \ General} \ \ell -\mathrm{leaky\ fort}$};
% Text Node
% \draw (373.67,2041.5) node [anchor=north west][inner sep=0.75pt]  [font=\small]  

% {$ \begin{array}{l}
% \mathrm{( b) \ Example} \ \mathrm{of\ a} \ 3-\mathrm{leaky\ fort}
% \end{array}$};

\end{tikzpicture}
    \caption{The figure on the left depicts a general fort in a graph $G$, highlighting the degree conditions. The figure on the right gives an example of a $3$-leaky fort, where the red vertices form the fort. If all vertices outside the fort are colored blue, then three leaks can prevent any forcing to occur.}
    \label{fig:LeakyForts}
\end{figure}

\section{Forts and leaky forcing on subgraphs of grid graphs}\label{sec:FortsInducedSubgraphs}
In this section, we study forts in subgraphs of the $d$-dimensional grid
$P_{n_1}\square\cdots\square P_{n_d}$ and use them to obtain bounds on
the $\ell$-leaky forcing number of these subgraphs. The geometric key here is that any nonempty set of vertices of degree $2d$ in the subgraph has neighbors outside the set, namely, one in each $\pm\mathbf e_i$ direction where $\mathbf e_i$ is the $i$th standard basis vector in $d$-dimensional space. We use this to show that when the number of leaks $\ell$ is fewer than $2d$, every nonempty $\ell$-leaky fort must contain a vertex of degree at most $2d-1$.

\begin{lem}\label{lem:2d-exposed-neighbors}
Let $d\ge 1$ and let $G = P_{n_1}\square \cdots \square P_{n_d}$. Let $H$ be a subgraph of $G$, and set
\[
R \;=\; \{v\in V(H): \deg_H(v)=2d\}.
\]
If $R'\subseteq R$ is a nonempty subset of vertices, then there exist at least $2d$ \emph{distinct} vertices
\[
y_1^+,y_1^-,\dots,y_d^+,y_d^- \in V(H)\setminus R'
\]
such that
\[
\bigl|N_H(y_i^\pm)\cap R'\bigr|=1 \qquad \text{for each } i\in\{1,\dots,d\}.
\]
\end{lem}

\begin{proof}
Let $G=P_{n_1}\square\cdots\square P_{n_d}$, and let $H$ be a subgraph of $G$. Set
\[
R=\{v\in V(H):\deg_H(v)=2d\},
\]
and let $R'\subseteq R$ be nonempty. Identify the vertices of $G$ with
\[
\{0,\ldots,n_1-1\}\times\cdots\times\{0,\ldots,n_d-1\}\subseteq\mathbb{Z}^d
\]
in the usual way, so that two vertices are adjacent in $G$ if and only if they differ by $\pm\mathbf{e}_i$ for some $i\in\{1,\ldots,d\}$, where $\mathbf{e}_i$ is the $i$th standard basis vector in $\mathbb{Z}^d$.

Let $i\in\{1,\dots,d\}$ and choose $x_i^+\in R'$ so that the $i$th coordinate $(x_i^+)_i$ is maximal among vertices in $R'$. Since $\deg_H(x_i^+)=2d$, all neighbors of $x_i^+$ must lie in $V(H)$, and in particular, the neighbor
\[
y_i^+ \;:=\; x_i^+ + \mathbf{e}_i
\]
belongs to $V(H)$. Moreover $y_i^+\notin R'$ by maximality of $(x_i^+)_i$.

We now claim that $N_H(y_i^+)\cap R'=\{x_i^+\}$. Assuming $w\in N_H(y_i^+)\cap R'$, we must have $w=y_i^+\pm \mathbf{e}_j$ for some $j$.
If $j=i$, then $w=y_i^+-\mathbf{e}_i=x_i^+$ or $w=y_i^++\mathbf{e}_i$. By the maximality of the $i$th coordinate $(x_i^+)_i$ in $R'$, we cannot have $w = y_i^++\mathbf{e}_i$, and thus, $w=y_i^+-\mathbf{e}_i=x_i^+$. If $j\neq i$, then $w$ has as its $i$th coordinate $(y_i^+)_i=(x_i^+)_i+1$, again
contradicting maximality of $x_i^+$. Hence, we get $N_H(y_i^+)\cap R'=\{x_i^+\}$.

Similarly, choose $x_i^-\in R'$ so that $(x_i^-)_i$ is minimal among vertices of $R'$. Since $\deg_H(x_i^-)=2d$, the neighbor
\[
y_i^- \;:=\; x_i^- - \mathbf{e}_i
\]
belongs to $V(H)$, satisfies $y_i^-\notin R'$, and an analogous argument gives $N_H(y_i^-)\cap R'=\{x_i^-\}$. Repeating this argument for each $i\in \{1, ... , d\}$ gives $2d$ vertices  $y_1^+,y_1^-,\dots,y_d^+,y_d^-$. 

Lastly, we will show that these $2d$ vertices are distinct.
We first show that $y_i^+\neq y_i^-$ for each $i\in\{1,\dots,d\}$. Recall that $y_i^+=x_i^++\mathbf{e}_i$ and
$y_i^-=x_i^--\mathbf{e}_i$. So, for each $i\in\{1,\dots,d\}$,
we have
\[
(y_i^+)_i=(x_i^+)_i+1
\qquad\text{and}\qquad
(y_i^-)_i=(x_i^-)_i-1,
\]
and since $(x_i^+)_i\ge (x_i^-)_i$, it follows that $(y_i^-)_i < (y_i^+)_i$. Hence $y_i^+\neq y_i^-$.

Now, we will show that $y_i^+$ and $y_i^-$
are distinct from both $y_j^+$ and $y_j^-$
whenever $i\neq j$. Since $y_j^\pm=x_j^\pm\pm\mathbf{e}_j$,
the $i$th coordinate is unchanged,  so
\[
(y_j^\pm)_i=(x_j^\pm)_i.
\]
Because $x_i^+$ and $x_i^-$ were chosen to have
maximal and minimal $i$th coordinates in $R'$,
respectively, and $x_j^\pm\in R'$, we have
\[
(x_i^-)_i\le (x_j^\pm)_i\le (x_i^+)_i.
\]
Therefore,
\[
(y_i^-)_i=(x_i^-)_i-1
<(y_j^\pm)_i
<(x_i^+)_i+1=(y_i^+)_i.
\]
Thus $y_j^\pm$ is distinct from both $y_i^-$ and $y_i^+$.
Since this holds for all distinct $i,j\in\{1,\dots,d\}$,
the vertices
\[
y_1^+,y_1^-,\dots,y_d^+,y_d^-
\]
are distinct.
    \end{proof}

Lemma~\ref{lem:2d-exposed-neighbors} shows that any nonempty set of
vertices of degree $2d$ in a subgraph $H$ of the $d$-dimensional grid
has at least $2d$ distinct vertices outside the set, each having
exactly one neighbor in the set. Let
\[
\delta H=\{v\in V(H):\deg_H(v)\leq 2d-1\}
\]
denote the set of vertices of $H$ with degree at most $2d-1$. We now prove that when $\ell\leq 2d-1$, every nonempty $\ell$-leaky fort of $H$ must intersect $\delta H$.

\begin{thm}\label{thm:lfort-hits-boundary-dgrid}
Let $d\geq 1$ and let $H$ be a subgraph of $P_{n_1}\square\cdots\square P_{n_d}$.  If $F$ is a \emph{nonempty}
$\ell$-leaky fort of $H$ with $\ell\le 2d-1$, then $F$ must contain a vertex of degree at most $2d-1$, i.e. $F \cap \delta H\neq \emptyset$.
\end{thm}

\begin{proof}
Let $R=\{v\in V(H):\deg_H(v)=2d\}$ where $H$ is a subgraph of the $d$-dimensional grid graph, and let $F$ be an $\ell$-leaky fort of $H$. Suppose, for the sake of contradiction, that $F\subseteq R$. By Lemma~\ref{lem:2d-exposed-neighbors},
there would be at least $2d$ distinct vertices in $V(H)\setminus F$, each having exactly one neighbor in $F$. This contradicts
Definition~\ref{def:lfort} since $\ell\leq 2d-1$. Hence, $F$ must contain at least one vertex of degree at most $2d-1$, so $F\cap\delta H\neq\emptyset$.
\end{proof}

\begin{cor}\label{cor:LeakyForcingBoundaryEquality}
    Let $H$ be a subgraph of $P_{n_1}\square \cdots \square P_{n_d}$, $d\ge 1$, and $0\le \ell \le 2d-1$. Then, we have 
    \[ |S_\ell| \le Z_{(\ell)}(H) \le |\delta H|\]
    where $S_\ell = \{v\in V(H) : \deg_H(v) \le \ell\}$. In particular, when $\ell = 2d-1$, we get $Z_{(\ell)}(H) = |\delta H|$. 
\end{cor}

\begin{proof}
For any $0\le \ell\le 2d-1$, Lemma~\ref{lem:lowdegree}
gives the lower bound $|S_\ell|\le Z_{(\ell)}(H)$. By Theorem~\ref{thm:lfort-hits-boundary-dgrid},
every $\ell$-leaky fort intersects $\delta H$ and thus, by Proposition~\ref{prop:LeakyForcingSetIntersectsForts},
$\delta H$ is an $\ell$-leaky forcing set. Therefore, $Z_{(\ell)}(H)\le |\delta H|$. Finally, when $\ell=2d-1$, we have
$S_\ell=\delta H$, and hence $Z_{(\ell)}(H)=|\delta H|$.
\end{proof}

\begin{cor}\label{cor:ClassifyLeaky}

Let $d\ge 1$ and $0\le \ell \le 2d-1$. For any subgraph $H$ of $P_{n_1}\square \cdots \square P_{n_d}$, if 
\[\deg_H(v) \in \{0,1,...,\ell\} \cup \{2d\} \ \text{ for all } v \in V(H),\]
then $Z_{(\ell)}(H) = |S_{\ell}|$ where $S_\ell = \{v\in V(H) : \deg_H(v) \le \ell\}$.
\end{cor}

\begin{proof}
This follows since $\delta H = S_{\ell}$ for any subgraph $H$ that satisfies the condition $\deg_H(v) \in \{0,1,...,\ell\} \cup \{2d\} \ \text{ for all } v \in V(H)$. Therefore, Corollary \ref{cor:LeakyForcingBoundaryEquality} gives
\[
|S_\ell|\leq Z_{(\ell)}(H)\leq|\delta H|=|S_\ell|,
\]
and hence $Z_{(\ell)}(H)=|S_\ell|$.

\end{proof}

Corollary~\ref{cor:ClassifyLeaky} is particularly useful for subgraphs
of grid graphs whose vertices have degree either $2d$ or at most $\ell$.
In these cases, the bounds established in
Section~\ref{sec:FortsInducedSubgraphs} give the exact value of
$Z_{(\ell)}(H)$. In the next section, we apply this result to determine
the $\ell$-leaky forcing numbers of Diamond rectangle graphs.

%========================
% Version B (boundary size at level 2d-1)
%========================

% \begin{thm}\label{thm:versionB-dgrid}
% Let $d\ge 1$ and let $H$ be an induced subgraph of $P_{n_1}\square \cdots \square P_{n_d}$. Define the (non-full-degree) boundary
% \[
% B \;=\; \{v\in V(H): \deg_H(v)\le 2d-1\}.
% \]
% Then
% \[
% Z_{(2d-1)}(H)=|B|.
% \]
% \end{thm}

% \begin{proof}
% By Lemma \ref{lem:lowdegree}, every $(2d-1)$-leaky forcing set contains all vertices of degree at most $2d-1$, i.e., it contains $B$.
% Hence
% \[
% Z_{(2d-1)}(H)\ge |B|.
% \]

% For the reverse inequality, let $F$ be a nonempty $(2d-1)$-leaky fort of $H$. By Theorem~\ref{thm:lfort-hits-boundary-dgrid},
% $F$ contains a vertex of degree at most $2d-1$, and therefore $F\cap B\neq \emptyset$. Thus $B$ intersects every $(2d-1)$-leaky fort
% of $H$, so Proposition~6 implies that $B$ is a $(2d-1)$-leaky forcing set. Consequently,
% \[
% Z_{(2d-1)}(H)\le |B|.
% \]
% Combining inequalities yields $Z_{(2d-1)}(H)=|B|$.
% \end{proof}

\section{Diamond rectangle graphs as induced subgraphs}\label{sec:ApplicationHopiRec}
We now utilize our previous results to analyze a specific family of induced subgraphs of the $2$-dimensional grid $P_{n_1}\square P_{n_2}$. In particular, we determine the zero forcing and leaky forcing numbers for the family of graphs
known as \emph{Aztec rectangle graphs}. These graphs arise as a natural generalization of Aztec diamond graphs, which were originally studied from a combinatorial perspective through the enumeration of domino tilings of the Aztec diamond \cite{FirstMentionAD,AztecDiamondGraphs-Benkart}. More recently, extensions to tilings of double Aztec rectangles have also been studied \cite{DoubleAztecRect}. From the graph-theoretic perspective, the equality of the zero forcing number and maximum nullity for Aztec diamond graphs was established in \cite{TechniquesLowerBoundZF} using
novel techniques. Motivated in part by this result, our aim in this section is to apply the general results of the previous
sections to the broader class of Aztec rectangle graphs. In doing so, we not only show that the zero forcing number and maximum nullity are equal for this family, but also give a complete characterization of the $\ell$-leaky forcing number for all $\ell\ge 1$.

A second motivation for our study of Aztec rectangle graphs was our curiosity about the origin of their name \textit{``Aztec Diamond/Rectangle graphs''}. The first mention of the term ``Aztec Diamond'' appeared in \cite{FirstMentionAD}, where the authors determined the number of ways one can tile the Aztec diamond using dominoes. Further, in a blog post \cite{ProppBlog}, one of the authors of that same paper reflected on the name choice and explained the following: 

\begin{center}
\textit{
``In writing up the work I’d done with Elkies, Kuperberg and Larsen \cite{FirstMentionAD}, I dubbed the shapes we’d studied “Aztec diamonds” because the design can be found in much pre-Columbian art.  I tried to figure out if there was a specific group of Native Americans most closely associated with the motif but it seemed to be shared between many nations. I eventually concluded that the Hopi made the most use of it, but “Diamond” didn’t sound as good as “Aztec diamond”, so I chose euphony over ethnographic accuracy.''}\end{center}

% From this post, the authors selected the name for this family of graphs based on its euphonic appeal rather than strict ethnographic accuracy. While the name may not perfectly align with its cultural or historical origins, and risks cultural appropriation as it repurposes a name with cultural significance without fully engaging with its historical or societal context, it was chosen for its ease of recognition within mathematical spaces. Equipped with this context, we now make the choice to use the term \textit{Hopi} in place of Aztec throughout the rest of the paper. 

In the cited post, the authors selected the name for this graph family based on euphonic appeal rather than strict ethnographic accuracy. While this terminology may not align with the cultural or historical origins of the design, it borrows a culturally significant name without accompanying historical context. To better understand this issue, we reached out to a professor of American Indian Studies, who noted that one of the designs resembled a pattern found in Southwest weaving, but did not recognize the others. They also noted that similar patterns could appear in other Indigenous communities, while acknowledging that they are not an expert in Hopi visual culture. Given this uncertainty, they suggested that replacing the name ``Aztec" with ``Hopi" may not be any more accurate and recommended using a different name altogether. Following this suggestion, and to avoid the risk of further cultural appropriation, we use the terms \emph{Diamond rectangle} and \emph{Diamond rectangle graph} throughout this paper. 

There are several ways to describe Diamond rectangles and their
corresponding Diamond rectangle graphs. One approach begins with
the \emph{diamond of order $n$}, a combinatorial object that can
be realized as the union of lattice squares described in
\cite{FirstMentionAD}. As noted in \cite{SPARSEMatching}, the
diamond graph of order $n$ can be viewed as the dual of the
diamond of order $n$. More precisely, the associated diamond
graph is obtained by taking the faces of these unit squares
as vertices, with two vertices adjacent whenever their
corresponding faces share an edge, as illustrated in
Figure~\ref{fig:Ad&ADgraph}. This construction extends to
the \emph{Diamond rectangle of order $(a,b)$}, which has
$a$ diagonal rows and $b$ diagonal columns of unit squares.
Since each unit square corresponds to a vertex in the
associated Diamond rectangle graph, $DR(a,b)$ can likewise
be viewed as having $a$ diagonal rows and $b$ diagonal
columns of vertices. This gives a visual interpretation
of the parameters $a$ and $b$.

Other descriptions of this graph family appear in the
literature. In \cite{KnuthAztecDiamond}, the authors considered
a graph on $mn$ vertices
$\{(x,y):1\le x\le m,\ 1\le y\le n\}$ that decomposes into
two subgraphs, denoted $EC_{m,n}$ and $OC_{m,n}$, and
identified the special case $OC_{2n+1,2n+1}$ with the
diamond graph of order $n$. Later,
\cite{VisualDefAztecDiamond} offered a more visual
perspective on these graphs. Motivated by these descriptions
and the techniques in \cite{TechniquesLowerBoundZF},
we define the Diamond rectangle graph $DR(a,b)$ as an
induced subgraph of the 2-dimensional grid graph
$P_{a+b}\square P_{a+b}$. This allows us to work directly
with the coordinates of its vertices and use the grid
structure throughout our proofs. We give this definition
below.

\begin{figure}[h!]
    \centering

\tikzset{every picture/.style={line width=0.75pt}} %set default line width to 0.75pt        

\begin{tikzpicture}[x=0.75pt,y=0.75pt,yscale=-1,xscale=1]
%uncomment if require: \path (0,3166); %set diagram left start at 0, and has height of 3166

%Shape: Ellipse [id:dp01904085319587523] 
\draw   (456.77,1593.93) .. controls (456.77,1591.37) and (458.86,1589.3) .. (461.44,1589.3) .. controls (464.02,1589.3) and (466.11,1591.37) .. (466.11,1593.93) .. controls (466.11,1596.48) and (464.02,1598.55) .. (461.44,1598.55) .. controls (458.86,1598.55) and (456.77,1596.48) .. (456.77,1593.93) -- cycle ;
%Shape: Ellipse [id:dp37159111248956744] 
\draw   (494.32,1593.93) .. controls (494.32,1591.37) and (496.41,1589.3) .. (498.99,1589.3) .. controls (501.57,1589.3) and (503.66,1591.37) .. (503.66,1593.93) .. controls (503.66,1596.48) and (501.57,1598.55) .. (498.99,1598.55) .. controls (496.41,1598.55) and (494.32,1596.48) .. (494.32,1593.93) -- cycle ;
%Straight Lines [id:da9581124404097807] 
\draw    (466.11,1593.93) -- (494.32,1593.93) ;
%Straight Lines [id:da8023645165409672] 
\draw    (461.44,1598.55) -- (461.44,1626.59) ;
%Straight Lines [id:da5512003550437814] 
\draw    (498.99,1598.55) -- (498.8,1626.59) ;
%Shape: Ellipse [id:dp1889931077399356] 
\draw   (456.77,1630.93) .. controls (456.77,1628.37) and (458.86,1626.3) .. (461.44,1626.3) .. controls (464.02,1626.3) and (466.11,1628.37) .. (466.11,1630.93) .. controls (466.11,1633.48) and (464.02,1635.55) .. (461.44,1635.55) .. controls (458.86,1635.55) and (456.77,1633.48) .. (456.77,1630.93) -- cycle ;
%Shape: Ellipse [id:dp8143048313704453] 
\draw   (494.32,1630.93) .. controls (494.32,1628.37) and (496.41,1626.3) .. (498.99,1626.3) .. controls (501.57,1626.3) and (503.66,1628.37) .. (503.66,1630.93) .. controls (503.66,1633.48) and (501.57,1635.55) .. (498.99,1635.55) .. controls (496.41,1635.55) and (494.32,1633.48) .. (494.32,1630.93) -- cycle ;
%Straight Lines [id:da49225027769716134] 
\draw    (466.11,1630.93) -- (494.32,1630.93) ;
%Shape: Ellipse [id:dp8912304799504929] 
\draw   (532.32,1630.93) .. controls (532.32,1628.37) and (534.41,1626.3) .. (536.99,1626.3) .. controls (539.57,1626.3) and (541.66,1628.37) .. (541.66,1630.93) .. controls (541.66,1633.48) and (539.57,1635.55) .. (536.99,1635.55) .. controls (534.41,1635.55) and (532.32,1633.48) .. (532.32,1630.93) -- cycle ;
%Straight Lines [id:da2381813614606989] 
\draw    (504.11,1630.93) -- (532.32,1630.93) ;
%Straight Lines [id:da9222015388292022] 
\draw    (499.44,1635.55) -- (499.44,1663.59) ;
%Straight Lines [id:da05653010242664691] 
\draw    (536.99,1635.55) -- (536.8,1663.59) ;
%Shape: Ellipse [id:dp35291465754070517] 
\draw   (494.77,1667.93) .. controls (494.77,1665.37) and (496.86,1663.3) .. (499.44,1663.3) .. controls (502.02,1663.3) and (504.11,1665.37) .. (504.11,1667.93) .. controls (504.11,1670.48) and (502.02,1672.55) .. (499.44,1672.55) .. controls (496.86,1672.55) and (494.77,1670.48) .. (494.77,1667.93) -- cycle ;
%Shape: Ellipse [id:dp5543350995441292] 
\draw   (532.32,1667.93) .. controls (532.32,1665.37) and (534.41,1663.3) .. (536.99,1663.3) .. controls (539.57,1663.3) and (541.66,1665.37) .. (541.66,1667.93) .. controls (541.66,1670.48) and (539.57,1672.55) .. (536.99,1672.55) .. controls (534.41,1672.55) and (532.32,1670.48) .. (532.32,1667.93) -- cycle ;
%Straight Lines [id:da3320441854608036] 
\draw    (504.11,1667.93) -- (532.32,1667.93) ;
%Shape: Ellipse [id:dp3679335748583392] 
\draw   (537.55,1598.83) .. controls (534.99,1598.79) and (532.95,1596.67) .. (532.99,1594.09) .. controls (533.03,1591.51) and (535.13,1589.45) .. (537.69,1589.49) .. controls (540.24,1589.53) and (542.28,1591.65) .. (542.24,1594.23) .. controls (542.2,1596.81) and (540.1,1598.87) .. (537.55,1598.83) -- cycle ;
%Straight Lines [id:da9515324089121131] 
\draw    (537.11,1627.04) -- (537.55,1598.83) ;
%Straight Lines [id:da09090527653917713] 
\draw    (541.67,1631.78) -- (569.7,1632.21) ;
%Straight Lines [id:da8989646448477504] 
\draw    (542.24,1594.23) -- (570.28,1594.85) ;
%Shape: Ellipse [id:dp2944149646093772] 
\draw   (573.96,1636.94) .. controls (571.41,1636.91) and (569.37,1634.78) .. (569.41,1632.2) .. controls (569.45,1629.62) and (571.55,1627.57) .. (574.11,1627.6) .. controls (576.66,1627.64) and (578.7,1629.77) .. (578.66,1632.35) .. controls (578.62,1634.93) and (576.52,1636.98) .. (573.96,1636.94) -- cycle ;
%Shape: Ellipse [id:dp20356554974707497] 
\draw   (574.54,1599.4) .. controls (571.99,1599.36) and (569.95,1597.23) .. (569.99,1594.66) .. controls (570.03,1592.08) and (572.13,1590.02) .. (574.68,1590.06) .. controls (577.24,1590.1) and (579.28,1592.22) .. (579.24,1594.8) .. controls (579.2,1597.38) and (577.1,1599.44) .. (574.54,1599.4) -- cycle ;
%Straight Lines [id:da11350481319187478] 
\draw    (574.11,1627.6) -- (574.54,1599.4) ;
%Shape: Ellipse [id:dp9881516134578121] 
\draw   (500.12,1561.25) .. controls (497.56,1561.21) and (495.52,1559.09) .. (495.56,1556.51) .. controls (495.6,1553.93) and (497.71,1551.87) .. (500.26,1551.91) .. controls (502.82,1551.95) and (504.86,1554.07) .. (504.82,1556.65) .. controls (504.78,1559.23) and (502.67,1561.29) .. (500.12,1561.25) -- cycle ;
%Straight Lines [id:da26909923805856084] 
\draw    (499.69,1589.46) -- (500.12,1561.25) ;
%Straight Lines [id:da49622962209931076] 
\draw    (504.24,1594.2) -- (532.27,1594.63) ;
%Straight Lines [id:da7534028269370452] 
\draw    (504.82,1556.65) -- (532.56,1557.08) ;
%Shape: Ellipse [id:dp8798846767837143] 
\draw   (537.11,1561.82) .. controls (534.56,1561.78) and (532.52,1559.66) .. (532.56,1557.08) .. controls (532.6,1554.5) and (534.7,1552.44) .. (537.26,1552.48) .. controls (539.81,1552.52) and (541.85,1554.64) .. (541.81,1557.22) .. controls (541.77,1559.8) and (539.67,1561.86) .. (537.11,1561.82) -- cycle ;
%Straight Lines [id:da8063536496900947] 
\draw    (536.68,1590.03) -- (537.11,1561.82) ;
%Straight Lines [id:da6606683005902503] 
\draw    (188.61,1577.03) -- (335.08,1577.37) ;
%Straight Lines [id:da45994481444788904] 
\draw    (188.94,1577.65) -- (188.94,1613.69) ;
%Straight Lines [id:da12909379613121907] 
\draw    (188.94,1614.15) -- (188.94,1650.19) ;
%Straight Lines [id:da7869808565297138] 
\draw    (188.94,1651.15) -- (188.94,1687.19) ;
%Straight Lines [id:da0199477540414168] 
\draw    (188.94,1687.65) -- (188.94,1723.69) ;
%Straight Lines [id:da7925707095899002] 
\draw    (188.61,1724.03) -- (224.82,1724.03) ;
%Straight Lines [id:da13419442850900865] 
\draw    (225.11,1724.03) -- (261.32,1724.03) ;
%Straight Lines [id:da8403535726203732] 
\draw    (152.6,1687.2) -- (298.04,1687.32) ;
%Straight Lines [id:da16560771064141866] 
\draw    (152.6,1651.6) -- (335.05,1651.28) ;
%Straight Lines [id:da7253157221313775] 
\draw    (152.6,1687.2) -- (153,1613.2) ;
%Straight Lines [id:da2730541034499703] 
\draw    (153,1613.2) -- (334.95,1613.58) ;
%Shape: Ellipse [id:dp8937276195771653] 
\draw   (457.32,1668.93) .. controls (457.32,1666.37) and (459.41,1664.3) .. (461.99,1664.3) .. controls (464.57,1664.3) and (466.66,1666.37) .. (466.66,1668.93) .. controls (466.66,1671.48) and (464.57,1673.55) .. (461.99,1673.55) .. controls (459.41,1673.55) and (457.32,1671.48) .. (457.32,1668.93) -- cycle ;
%Straight Lines [id:da3947073995717243] 
\draw    (461.44,1635.55) -- (461.44,1663.59) ;
%Straight Lines [id:da2620728627863276] 
\draw    (466.11,1667.93) -- (494.32,1667.93) ;
%Shape: Ellipse [id:dp7049269965027793] 
\draw   (420.32,1669.93) .. controls (420.32,1667.37) and (422.41,1665.3) .. (424.99,1665.3) .. controls (427.57,1665.3) and (429.66,1667.37) .. (429.66,1669.93) .. controls (429.66,1672.48) and (427.57,1674.55) .. (424.99,1674.55) .. controls (422.41,1674.55) and (420.32,1672.48) .. (420.32,1669.93) -- cycle ;
%Shape: Ellipse [id:dp5715950498897249] 
\draw   (419.77,1631.93) .. controls (419.77,1629.37) and (421.86,1627.3) .. (424.44,1627.3) .. controls (427.02,1627.3) and (429.11,1629.37) .. (429.11,1631.93) .. controls (429.11,1634.48) and (427.02,1636.55) .. (424.44,1636.55) .. controls (421.86,1636.55) and (419.77,1634.48) .. (419.77,1631.93) -- cycle ;
%Straight Lines [id:da10340562293058508] 
\draw    (429.11,1631.93) -- (457.32,1631.93) ;
%Straight Lines [id:da06745250216058274] 
\draw    (424.44,1636.55) -- (424.44,1664.59) ;
%Straight Lines [id:da27600709236179655] 
\draw    (429.11,1668.93) -- (457.32,1668.93) ;
%Shape: Ellipse [id:dp7217942064022724] 
\draw   (457.32,1706.93) .. controls (457.32,1704.37) and (459.41,1702.3) .. (461.99,1702.3) .. controls (464.57,1702.3) and (466.66,1704.37) .. (466.66,1706.93) .. controls (466.66,1709.48) and (464.57,1711.55) .. (461.99,1711.55) .. controls (459.41,1711.55) and (457.32,1709.48) .. (457.32,1706.93) -- cycle ;
%Straight Lines [id:da4374133080546154] 
\draw    (461.44,1673.55) -- (461.44,1701.59) ;
%Straight Lines [id:da9672309312187011] 
\draw    (498.99,1673.55) -- (498.8,1701.59) ;
%Shape: Ellipse [id:dp29420624078095425] 
\draw   (494.32,1705.93) .. controls (494.32,1703.37) and (496.41,1701.3) .. (498.99,1701.3) .. controls (501.57,1701.3) and (503.66,1703.37) .. (503.66,1705.93) .. controls (503.66,1708.48) and (501.57,1710.55) .. (498.99,1710.55) .. controls (496.41,1710.55) and (494.32,1708.48) .. (494.32,1705.93) -- cycle ;
%Straight Lines [id:da19852907123271657] 
\draw    (466.11,1705.93) -- (494.32,1705.93) ;
%Straight Lines [id:da8695189800128807] 
\draw    (225.18,1539.62) -- (225.11,1724.03) ;
%Straight Lines [id:da22606257599411872] 
\draw    (261.78,1539.89) -- (261.59,1723.93) ;
%Straight Lines [id:da6831060045408865] 
\draw    (298.39,1540.16) -- (298.04,1687.32) ;
%Straight Lines [id:da006281760138264314] 
\draw    (225.18,1539.62) -- (298.39,1540.16) ;
%Straight Lines [id:da7575848704289305] 
\draw    (335.08,1577.37) -- (334.95,1613.58) ;
%Straight Lines [id:da9094189281464261] 
\draw    (334.95,1613.58) -- (335.05,1651.28) ;

% Text Node
\draw (90,1733.4) node [anchor=north west][inner sep=0.75pt]  [font=\normalsize] [align=left] {(a) A Diamond rectangle of order (2,3)};
% Text Node
\draw (375,1732.6) node [anchor=north west][inner sep=0.75pt]   [align=left] {(b) A Diamond rectangle graph of order (2,3)};

\end{tikzpicture}
    \caption{}
    \label{fig:Ad&ADgraph}
\end{figure}

\begin{defn}\label{def:UpdatedDefAD_{m,n}}
    Let $a,b \in \mathbb{N}$ and consider the 2-dimensional grid graph given by $P_{a+b} \square P_{a+b}$, with each copy of $P_{a+b}$ indexed by $\{0,\ldots,a+b-1\}$. The \textit{Diamond rectangle graph of order $(a,b)$}, denoted by $DR(a,b)$, is the induced subgraph of $P_{a+b}\square P_{a+b}$ determined by the following subset of vertices
    $$J = \{(i,j): a-1 \leq i+j \leq 2b+a-1 \text{ and } |i-j| \leq a\}.$$

    \noindent That is, $DR(a,b) := (P_{a+b}\square P_{a+b})[J]$.

\end{defn} 

From Definition~\ref{def:UpdatedDefAD_{m,n}}, we can view
$DR(a,b)$ as a subgraph of the $2$-dimensional grid graph
$P_{a+b}\square P_{a+b}$ by removing vertices from the four
corners of the grid. This representation was useful both for
implementing our work in SageMath, particularly in applying
the algorithm described in \cite{moruzziLeakyForcing},
and for establishing several of our results by viewing the
graph as embedded in an integer lattice. The following theorem makes this representation precise by describing the vertices removed from each corner, the resulting number of vertices in $DR(a,b)$, and the degrees of the vertices of $DR(a,b)$.

\begin{thm}\label{thm:DR-visual}
Let $a,b\ge 1$. The Diamond rectangle graph $DR(a,b)$ can be obtained
from the grid graph $P_{a+b}\square P_{a+b}$
by removing the vertices $(i,j)$ satisfying one of the following:
\[
i+j<a-1,\qquad
i+j>2b+a-1,\qquad
i-j>a,\qquad
j-i>a.
\]
Moreover, $|V(DR(a,b))|=2ab+a+b$, and the vertices satisfying the following
\[
i+j=a-1,\qquad i+j=2b+a-1,\qquad i-j=a,\qquad j-i=a
\]
have degree $2$, while all other vertices of $DR(a,b)$ have degree $4$.
\end{thm}

\begin{proof}
Identify the vertices of $P_{a+b}\square P_{a+b}$ with
\[
\{0,\ldots,a+b-1\}\times\{0,\ldots,a+b-1\}.
\]
By Definition \ref{def:UpdatedDefAD_{m,n}}, a vertex $(i,j)$ belongs to $DR(a,b)$
if and only if
\[
a-1\le i+j\le 2b+a-1
\qquad\text{and}\qquad
|i-j|\le a.
\]
Thus the vertices not in $DR(a,b)$ are precisely those $(i,j)$ satisfying one of
\[
i+j<a-1,\qquad
i+j>2b+a-1,\qquad
i-j>a,\qquad
j-i>a.
\]
Moreover, adjacency of vertices in $DR(a,b)$ is the adjacency of vertices from
$P_{a+b}\square P_{a+b}$, so removing precisely these vertices from the
grid graph gives $DR(a,b)$.

For the condition $i+j<a-1$, since both $i$ and $j$ are bounded below by $0$, the possible values of $i+j$ are $0,1,\ldots,a-2$. Since $i$ increases from left to right and $j$ increases from bottom to top, these values correspond to the bottom left corner of the grid graph.

Fix $k\in\{0,\ldots,a-2\}$. A vertex satisfying $i+j=k$ is determined by choosing $i$ with
\[
0\le i\le k,
\]
since then $j=k-i$. Thus, there are $k+1$ vertices satisfying $i+j=k$. As $k$ ranges from $0$ to $a-2$, the numbers of vertices are $1,2,\ldots,a-1$. Therefore, summing, the bottom left corner contains
\[
1+2+\cdots+(a-1)
=\frac{a(a-1)}{2}
\]
vertices.

For the condition $i+j>2b+a-1$, since both $i$ and $j$ are bounded above by $a+b-1$, the possible values of $i+j$ are $2b+a,\ldots,2a+2b-2$. Since $i$ increases from left to right and $j$ increases from bottom to top, these values correspond to the upper right corner of the grid graph.

Fix $k\in\{2b+a,\ldots,2a+2b-2\}$. A vertex satisfying $i+j=k$ is determined by choosing $i$, with
\[
k-(a+b-1)\le i\le a+b-1,
\]
since then $j=k-i$. Thus there are
\[
(a+b-1)-\bigl(k-(a+b-1)\bigr)+1
=2(a+b)-1-k
\]
vertices satisfying $i+j=k$. As $k$ ranges from $2b+a$ to $2a+2b-2$, the numbers of vertices are $a-1,a-2,\ldots,1$. Therefore, summing, the upper right corner again contains $\frac{a(a-1)}{2}$ vertices.

For the condition $j-i>a$, since $j$ is bounded above by $a+b-1$ and $i$ is bounded below by $0$, the possible values of $j-i$ are $a+1,\ldots,a+b-1$. Since $i$ increases from left to right and $j$ increases from bottom to top, these values correspond to the upper left corner of the grid graph.

Fix $k\in\{a+1,\ldots,a+b-1\}$. A vertex satisfying $j-i=k$ is determined by choosing $i$, with
\[
0\le i\le a+b-1-k,
\]
since then $j=i+k$. Thus, there are $a+b-k$ vertices satisfying $j-i=k$. As $k$ ranges from $a+1$ to $a+b-1$, the numbers of vertices are $b-1,b-2,\ldots,1$. Therefore, summing, the upper left corner contains
\[
(b-1)+(b-2)+\cdots+1
=\frac{b(b-1)}{2}
\]
vertices.

For the condition $i-j>a$, since $i$ is bounded above by $a+b-1$ and $j$ is bounded below by $0$, the possible values of $i-j$ are $a+1,\ldots,a+b-1$. Since $i$ increases from left to right and $j$ increases from bottom to top, these values correspond to the bottom right corner of the grid graph.

Fix $k\in\{a+1,\ldots,a+b-1\}$. A vertex satisfying $i-j=k$ is determined by choosing $j$, with
\[
0\le j\le a+b-1-k,
\]
since then $i=j+k$. Thus there are $a+b-k$ vertices satisfying $i-j=k$. As $k$ ranges from $a+1$ to $a+b-1$, the numbers of vertices are $b-1,b-2,\ldots,1$. Therefore, summing, the bottom right corner also contains $\frac{b(b-1)}{2}$ vertices.

Note, the four corner regions are disjoint. The first two conditions are clearly
disjoint, as are the last two. Also, if $i+j<a-1$, then neither $i-j>a$
nor $j-i>a$ can hold. For example, if $i-j>a$, then
\[
i\ge j+a+1,
\]
and hence
\[
i+j\ge a+1+2j\ge a+1,
\]
a contradiction. The case $j-i>a$ is similar.

Likewise, if $i+j>2b+a-1$, then neither $i-j>a$ nor $j-i>a$ can hold.
For example, if $i-j>a$, then
\[
j\le i-a-1\le b-2,
\]
and hence
\[
i+j\le (a+b-1)+(b-2)<2b+a-1,
\]
a contradiction. The case $j-i>a$ is similar.

Therefore, the total number of removed vertices is $a(a-1)+b(b-1).$ Since $|P_{a+b}\square P_{a+b}|=(a+b)^2$, it follows that
\[
|V(DR(a,b))|
=(a+b)^2-a(a-1)-b(b-1)
=2ab+a+b.
\]

It remains to show that every vertex of $DR(a,b)$ has degree either $2$ or $4$.
Let $v=(i,j)\in V(DR(a,b))$. If all four defining inequalities are strict, so that
\[
a-1<i+j<2b+a-1
\qquad\text{and}\qquad
|i-j|<a,
\]
then each of the four grid neighbors
\[
(i+1,j),\qquad (i-1,j),\qquad (i,j+1),\qquad (i,j-1)
\]
also satisfies the defining inequalities for $DR(a,b)$. Thus $\deg_{DR(a,b)}(v)=4$.

Now suppose $v$ lies on one of the four boundary lines
\[
i+j=a-1,\qquad
i+j=2b+a-1,\qquad
j-i=a,\qquad
i-j=a.
\]
These four boundary conditions are disjoint. The first two conditions and
the last two conditions are clearly disjoint. Also, a vertex cannot
satisfy one of the first two conditions and one of the last two conditions.
For example, if
\[
i+j=a-1
\qquad\text{and}\qquad
i-j=a,
\]
then $2i=2a-1$, which is impossible since $i$ is an integer. The other
cases are similar. Thus, every boundary vertex lies on exactly one of
these four boundary lines.

If $i+j=a-1$, then the neighbors $(i-1,j)$ and $(i,j-1)$
do not belong to $DR(a,b)$, while $(i+1,j)$ and $(i,j+1)$ do.
Hence $v$ has degree $2$. Similarly, if $i+j=2b+a-1$,
then $(i-1,j)$ and $(i,j-1)$ belong to $DR(a,b)$,
while $(i+1,j)$ and $(i,j+1)$ do not.
Thus, $v$ again has degree $2$.

If $j-i=a$, then $(i+1,j)$ and $(i,j-1)$ belong to
$DR(a,b)$, while $(i-1,j)$ and $(i,j+1)$ do not.
Similarly, if $i-j=a$, then $(i-1,j)$ and $(i,j+1)$
belong to $DR(a,b)$, while $(i+1,j)$ and $(i,j-1)$
do not. Therefore, every boundary vertex has degree $2$,
and hence every vertex of $DR(a,b)$ has degree either $2$ or $4$.
\end{proof}

It is a known observation that if an edge covering of a graph $G$ is determined by subgraphs $H_1,\ldots,H_h$, then the minimum rank of the graph $G$ is bounded above (and hence the maximum nullity bounded below) by the sum of the minimum ranks of the subgraphs $H_i$ which compose its edge covering. That is,
\begin{obs}[Observation 1.7, \cite{FallatHogbenObservation}]\label{obs:lowerboundmr}
    If $G = \bigcup_{i=1}^h H_i$, then $\operatorname{mr}(G)\leq \sum_{i=1}^h \operatorname{mr}(H_i)$.
\end{obs}

\noindent where $E(G) = \bigcup_i^h E(H_i)$. Utilizing this observation, we have the following two propositions that will be useful in establishing a lower bound on the maximum nullity of a given Diamond rectangle graph.

\begin{prop}\label{prop:edgecoverbounds}
    Let $G$ be any graph. If $G$ admits an edge covering by $\alpha$ copies of the cycle graph $C_4$, then $\operatorname{M}(G) \ge |V(G)| - 2\cdot \alpha$.
\end{prop}

\begin{proof}
Let $G$ be a graph that admits an edge covering by $\alpha$ copies of $C_4$. By Observation \ref{obs:lowerboundmr}, we have \[\operatorname{M}(G) = |V(G)| - \operatorname{mr}(G) \ge |V(G)|  - \sum\limits_{i=1}^\alpha \operatorname{mr}(C_4).\] Since $\operatorname{mr}(C_4) = 2$, the result follows.
\end{proof}

Now, in terms of the Diamond rectangle graph of order $(a,b)$, we are able to show there exists an edge cover consisting of $ab$ copies of the cycle graph $C_4$.

\begin{prop}\label{prop:Edge Cover HD(m,n)}
For any $a,b\in\mathbb{N}$, the edges of the Diamond rectangle graph
$DR(a,b)$ can be covered by $ab$ copies of $C_4$. In particular, there
exist copies
\[
C_4^{(r,s)}, \qquad 0\le r\le b-1,\quad 0\le s\le a-1,
\]
such that
\[
E(DR(a,b))
=
\bigcup_{r=0}^{b-1}\bigcup_{s=0}^{a-1}
E\bigl(C_4^{(r,s)}\bigr).
\]
\end{prop}

\begin{proof}
By Definition \ref{def:UpdatedDefAD_{m,n}}, a vertex $(i,j)$ belongs to $DR(a,b)$
if and only if
\[
a-1\leq i+j\leq 2b+a-1
\qquad\text{and}\qquad
|i-j|\leq a.
\]
In view of the visual representation in Theorem \ref{thm:DR-visual}, we
construct $ab$ copies of $C_4$ by indexing the unit squares of $DR(a,b)$
in the two diagonal directions determined by $i+j$ and $i-j$. For $0\le r\le b-1$ and $0\le s\le a-1$, let $C_4^{(r,s)}$ be the 4-cycle with vertices
\[
(r+s,a+r-s),\qquad
(r+s+1,a+r-s),
\]
\[
(r+s,a-1+r-s),\qquad
(r+s+1,a-1+r-s).
\]
Thus, $C_4^{(r,s)}$ is the unit square in the grid whose lower-left
vertex is $(r+s,a-1+r-s)$.  For the four vertices of $C_4^{(r,s)}$, notice the values of $i+j$ are
\[
a-1+2r,\qquad a+2r,\qquad a+2r,\qquad a+1+2r.
\]
Since $0\le r\le b-1$, each of these values lies between $a-1$ and $2b+a-1$. Similarly, the values of $i-j$ for the four vertices are
\[
2s-(a-1),\qquad 2s-(a-2),\qquad 2s-a,\qquad 2s-(a-1).
\]
Since $0\le s\le a-1$, these values lie between $-a$ and $a$. Thus each of the four vertices satisfies
\[
a-1\le i+j\le 2b+a-1
\qquad\text{and}\qquad
|i-j|\le a.
\]
Therefore, for all $0\le r \le b-1$ and $0\le s \le a-1$, each $C_4^{(r,s)}$ is a subgraph of $DR(a,b)$. Since there are $b$ choices for $r$ and $a$ choices for $s$, this gives $ab$ copies of $C_4$.

We now show that these $ab$ copies of $C_4$ cover every edge of
$DR(a,b)$. First, we show that no two copies of $C_4$ share an edge. Note that the lower-left vertex $(p,q)$ of each $C_4^{(r,s)}$ satisfies
\[
p+q=a-1+2r.
\]
Suppose two distinct copies $C_4^{(r,s)}$ and
$C_4^{(r',s')}$ share an edge, and let $(p,q)$ and
$(p',q')$ be their lower-left vertices. Since two unit squares
that share an edge have lower-left vertices that differ by either
$(1,0)$ or $(0,1)$, we must have
\[
|(p+q)-(p'+q')|=1.
\]
On the other hand,
\[
(p+q)-(p'+q')
=(a-1+2r)-(a-1+2r')
=2(r-r'),
\]
which is even. This is a contradiction. Therefore, no two distinct
copies $C_4^{(r,s)}$ share an edge. Hence, the $ab$ copies contain
a total of $4ab$ distinct edges.

It remains to show that $DR(a,b)$ has exactly $4ab$ edges. Consider an edge joining $(i,j)$ and $(i+1,j)$. By definition, both endpoints lie in $DR(a,b)$ if and only if
\[
a-1\le i+j\le 2b+a-2
\qquad\text{and}\qquad
-a\le i-j\le a-1.
\]
There are $2b$ possible values of $i+j$ and $2a$ possible values of
$i-j$. Since $i+j$ and $i-j$ must have the same parity, exactly half of these pairs determine integer values of $i$ and $j$. Hence there are $2ab$ edges joining $(i,j)$ and $(i+1,j)$.

By symmetry in the two coordinates, there are also $2ab$ edges joining
$(i,j)$ and $(i,j+1)$. Therefore,
\[
|E(DR(a,b))|=4ab.
\]
Since the $ab$ copies $C_4^{(r,s)}$ are subgraphs of
$DR(a,b)$ that do not share any edges and together contain $4ab$ edges, they cover every edge of
$DR(a,b)$. Hence
\[
E(DR(a,b))
=
\bigcup_{r=0}^{b-1}\bigcup_{s=0}^{a-1}
E\bigl(C_4^{(r,s)}\bigr).
\]
\end{proof}

Combining Proposition \ref{prop:edgecoverbounds} and Proposition \ref{prop:Edge Cover HD(m,n)}, we get an upper bound on the minimum rank of $DR(a,b)$ and subsequently a lower bound on the maximum nullity of $DR(a,b)$.

\begin{cor}\label{cor:Min Rank HD(m,n)}
    The minimum rank of the Diamond rectangle graph $DR(a,b)$ of order $(a,b)$ is bounded above by $2ab$, i.e. $$\operatorname{mr}(DR(a,b))\leq 2ab.$$
Moreover, the maximum nullity of $DR(a,b)$ is bounded below by $a+b$. 
\end{cor}

% \begin{proof}
%     We note that the graph $HD(m,n)$ admits an edge covering of $mn$ copies of $C_4$. Thus, we get \[ M(HD(m,n) \ge |V(HD(m,n))| - 2m n\]
%     The second part follows by a direct application of the rank-nullity theorem together with the order of $HD(m,n)$ given in Proposition \ref{prop: Order of HD(m,n)}.
% \end{proof}

\section{Zero forcing and leaky forcing on Diamond rectangle graphs}\label{sec:ZFonHopiRectanlges}
In this section, we determine the zero forcing number and the $\ell$-leaky forcing numbers for Diamond rectangle graphs $DR(a,b)$. We first establish the zero forcing number and show that it is equal to the maximum nullity for this graph family. We then determine the $\ell$-leaky forcing numbers. For $\ell\ge 2$, we apply the results from Section~\ref{sec:FortsInducedSubgraphs} to obtain the exact values of $Z_{(\ell)}(DR(a,b))$. For $\ell=1$, we instead use Theorem~\ref{prop:twoways} to determine $Z_{(1)}(DR(a,b))$. These results also suggest that studying $1$-leaky forcing for induced subgraphs of $d$-dimensional grid graphs would be a natural direction for future work.

\begin{lem}\label{lem:ZF Upper Bound}
Let $a,b \in \mathbb{N}$ and consider the following subset of vertices of $DR(a,b)$:
\begin{equation}\label{eq:ZeroForcingSet}
    B=\{(i,j)\in V(DR(a,b)): i+j=a-1\}\cup\{(i,j)\in V(DR(a,b)):j=i+a\}.
\end{equation}
Then, the set $B$ has cardinality $a+b$ and is a zero forcing set of $DR(a,b)$.
\end{lem}

\begin{proof}
Let
\[
B_1=\{(i,j)\in V(DR(a,b)):i+j=a-1\}, \qquad
B_2=\{(i,j)\in V(DR(a,b)):j=i+a\},
\]
so that $B=B_1\cup B_2$. By Definition \ref{def:UpdatedDefAD_{m,n}}, the vertices of $DR(a,b)$ satisfy
\[
a-1\leq i+j\leq 2b+a-1
\qquad\text{and}\qquad
|i-j|\leq a.
\]
The vertices satisfying $i+j=a-1$ are
\[
(0,a-1),(1,a-2),\ldots,(a-1,0),
\]
so $|B_1|=a$. Similarly, the vertices satisfying $j=i+a$ are
\[
(0,a),(1,a+1),\ldots,(b-1,a+b-1),
\]
so $|B_2|=b$. If $(i,j)\in B_1\cap B_2$, then $i+(i+a)=a-1$, which gives $2i=-1$, a contradiction since $i$ is an integer. Therefore,
$B_1\cap B_2=\emptyset$, and
\[
|B|=|B_1|+|B_2|=a+b.
\]

We now show that $B$ is a zero forcing set by forcing horizontally
from left to right. First, we note that $B$ consists precisely of the
vertices of $DR(a,b)$ that do not have a neighbor immediately to their
left. For $(i,j)\in V(DR(a,b))$, the vertex $(i-1,j)$ lies in
$DR(a,b)$ if and only if
\[
i+j\geq a
\qquad\text{and}\qquad
j-i\leq a-1.
\]
Thus, $(i-1,j)\notin V(DR(a,b))$ precisely when
\[
i+j=a-1
\qquad\text{or}\qquad
j-i=a,
\]
which are exactly the conditions defining $B_1$ and $B_2$. Therefore,
every vertex in $V(DR(a,b))\setminus B$ has a neighbor immediately to
its left.

Now, consider the collection of horizontal forces
\[
F_1=\{(i-1,j)\to(i,j):(i,j)\in V(DR(a,b))\setminus B\}.
\]
We will show that $F_1$ admits a chronological ordering starting from $B$ that turns every vertex of $DR(a,b)$ blue. Thus, $F_1$ is a forcing process starting from $B$. Starting
with every vertex in $B$ colored blue, we proceed by induction on the first coordinate. Every vertex with first coordinate $0$ belongs to $B$, so all such vertices are initially blue.

Now suppose that all vertices with first coordinate at most $k$ are blue,
and consider a vertex $(k+1,j)\in V(DR(a,b))$. If $(k+1,j)\in B$, then
it is already blue. Otherwise, because $(k+1,j)\notin B$, its left neighbor $(k,j)$ lies in $V(DR(a,b))$. Every neighbor of $(k,j)$ other than
$(k+1,j)$ has first coordinate at most $k$ and is therefore blue. Hence, if $(k+1,j)$ is white, then it is the unique white neighbor of $(k,j)$, so
\[
(k,j)\to(k+1,j)
\]
is a valid force in $F_1$. Therefore, every vertex with first coordinate $k+1$ becomes blue, completing the induction step. By induction on the first coordinate, every vertex of $DR(a,b)$ is eventually colored blue. Hence, $B$ is a zero forcing set.
\end{proof}
\begin{figure}[h!]
    \centering

\tikzset{every picture/.style={line width=0.75pt}} %set default line width to 0.75pt        



\caption{The first three rounds of the forcing process $F_1$ on $DR(2,3)$, illustrating the horizontal forces across consecutive columns, shown from left to right.}
    \label{fig:Zfexample}
\end{figure}

The following establishes the values of both the maximum nullity and the zero forcing number for the family of Diamond rectangle graphs.

\begin{thm}\label{thm: ZF Number of HD(m,n)}
For all $a,b \in \mathbb{N}$, the zero forcing number of the Diamond rectangle graph $DR(a,b)$ is $a+b$. In particular, the maximum nullity and the zero forcing number are equal for this family of graphs.
\end{thm}

\begin{proof}
From Lemma $\ref{lem:ZF Upper Bound}$, the set $B$ given in \eqref{eq:ZeroForcingSet} forms a zero forcing set of $DR(a,b)$ of size $a+b$ and thus, $Z(DR(a,b)) \leq a+b.$ On the other hand, by Corollary \ref{cor:Min Rank HD(m,n)}, we have that 
\[
a+b \le \operatorname{M}(DR(a,b)) .
\]
Since maximum nullity is always bounded above by the zero forcing number as in Proposition \ref{thm:NullityInequality}, we have
\[
a+b \leq \operatorname{M}(DR(a,b)) \leq Z(DR(a,b)) \leq a+b,
\]
which forces equality.

\end{proof}

To determine the $1$-leaky forcing number of $DR(a,b)$, we use
Theorem \ref{prop:twoways} and show that every vertex outside $B$ can be forced in two distinct ways. By Lemma \ref{lem:ZF Upper Bound}, the horizontal forcing process $F_1$ starting from the initial blue set $B$ in \eqref{eq:ZeroForcingSet} gives one force into every vertex outside $B$. We now define two additional sets of forces, $F_2$ and $F_3$, and then show that each admits a chronological ordering starting from $B$. Together with $F_1$, these forces will give two distinct ways to force every vertex outside $B$.

In the definitions below, we include only those indicated forces for which both the forcing vertex and the target vertex belong to
$V(DR(a,b))$.

\begin{equation}\label{eq:F2Process}
F_{2}=
\left\{
\begin{array}{lll}
(i,j)\rightarrow(i+1,j),
& i+j=a-1+2k,
& 0\leq k\leq b-1,\\[6pt]
(i,j)\rightarrow(i,j+1),
& i+j=a+2k,
& 0\leq k\leq b-1.
\end{array}
\right\}
\end{equation}

\begin{equation}\label{eq:F3Process}
F_{3}=
\left\{
\begin{array}{lll}
(i,j)\rightarrow(i+1,j),
& j=i+a-2k,
& 0\leq k\leq a-1,\\[6pt]
(i,j)\rightarrow(i,j-1),
& j=i+a-(2k+1),
& 0\leq k\leq a-1.
\end{array}
\right\}.
\end{equation}

The next proposition shows that both $F_2$ and $F_3$ admit chronological orderings starting from $B$ and therefore define forcing processes on $DR(a,b)$.

\begin{prop}\label{prop:F2F3}
The sets of forces $F_2$ and $F_3$ defined in \eqref{eq:F2Process} and \eqref{eq:F3Process} each admit a
chronological ordering starting from $B$ that forces every vertex outside of
$B$ blue.
\end{prop}

\begin{proof}
We first consider $F_2$ defined in \eqref{eq:F2Process}. The forces in $F_2$ proceed through consecutive diagonals of the form $i+j=c$, beginning with the diagonal $i+j=a-1$, which is contained in $B$. At each stage, horizontal forces color the next diagonal, followed by vertical forces that color the subsequent diagonal. Proceeding by induction on $k$, we will show that all vertices on the diagonals
\[
i+j\leq a-1+2k
\]
are blue. For $k=0$, the diagonal
$i+j=a-1$ is contained in $B$ and is therefore blue.

Now suppose that all vertices on the diagonals
\[
i+j\leq a-1+2k
\]
are blue for some $0\leq k\leq b-1$. We will show that the
forces in $F_2$ can be performed in an order that colors the
next two diagonals blue.

Consider the vertex $(k,a-1+k)$ on the diagonal
$i+j=a-1+2k$. Its upper neighbor $(k,a+k)$ belongs to $B$,
while its left and lower neighbors, when they exist, lie on
the preceding diagonal and are already blue. Thus, its only
white neighbor is the vertex immediately to its right, so
it can perform the horizontal force
\[
(k,a-1+k)\to(k+1,a-1+k).
\]
The newly blue vertex is then the upper neighbor of the next
forcing vertex on the diagonal. Since the left and lower
neighbors of each successive forcing vertex are already blue,
the horizontal forces can be performed successively from upper
left to lower right. These forces color every vertex on the
diagonal $i+j=a+2k$, since its remaining vertex $(k,a+k)$
already belongs to $B$.

We next perform the vertical forces from this diagonal.
The vertex $(a+k,k)$ lies on the diagonal $i+j=a+2k$
and has no neighbor immediately to its right in $DR(a,b)$.
Its left and lower neighbors, when they exist, are already
blue. Thus, its only white neighbor is the vertex immediately
above it, so it can force upward. The newly blue vertex is
then the right neighbor of the next forcing vertex. Since
the left and lower neighbors of each successive forcing
vertex are already blue, the vertical forces can be performed
successively from lower right to upper left, making the
entire diagonal $i+j=a+1+2k$ blue.

Therefore, all vertices on the diagonals
\[
i+j\leq a-1+2(k+1)
\]
are blue, completing the induction step. Applying the induction step through $k=b-1$, all vertices satisfying
\[
i+j\leq a-1+2b=2b+a-1
\]
are blue. By Definition \ref{def:UpdatedDefAD_{m,n}}, every vertex of $DR(a,b)$ satisfies this inequality. Hence, $F_2$ admits a chronological ordering starting from $B$ that forces every vertex of $DR(a,b)$ blue. The first three rounds of this process are shown in Figure \ref{fig:F2Process}.

We now consider $F_3$ defined in \eqref{eq:F3Process}. The forces in $F_3$ proceed through consecutive diagonals of the form $j-i=c$, beginning with the diagonal $j-i=a$, which is contained in $B$. At each stage, horizontal forces color the next diagonal, followed by vertical forces that color the subsequent diagonal. Proceeding by induction on $k$, we will show that all vertices on the diagonals
\[
j-i\geq a-2k
\]
are blue. For $k=0$, the diagonal
$j-i=a$ is contained in $B$ and is therefore blue.

Now suppose that all vertices on the diagonals
\[
j-i\geq a-2k
\]
are blue for some $0\leq k\leq a-1$. We will show that
the forces in $F_3$ can be performed in an order that
colors the next two diagonals blue.

Consider the vertex $(k,a-k)$ on the diagonal
$j-i=a-2k$. Its lower neighbor $(k,a-k-1)$ belongs to
$B$, while its left and upper neighbors, when they exist,
lie on the preceding diagonal and are already blue.
Thus, its only white neighbor is the vertex immediately
to its right, so it can perform the horizontal force
\[
(k,a-k)\to(k+1,a-k).
\]
The newly blue vertex is then the lower neighbor of the
next forcing vertex on the diagonal. Since the left and
upper neighbors of each successive forcing vertex are
already blue, the horizontal forces can be performed
successively from lower left to upper right. These forces
color every vertex on the diagonal $j-i=a-(2k+1)$, since
its remaining vertex $(k,a-k-1)$ already belongs to $B$.

We next perform the vertical forces from this diagonal.
The vertex $(b+k,a+b-1-k)$ lies on the diagonal
$j-i=a-(2k+1)$ and has no neighbor immediately to its
right in $DR(a,b)$. Its left and upper neighbors, when
they exist, are already blue. Thus, its only white
neighbor is the vertex immediately below it, so it can
force downward. The newly blue vertex is then the right
neighbor of the next forcing vertex. Since the left and
upper neighbors of each successive forcing vertex are
already blue, the vertical forces can be performed
successively from upper right to lower left, making the
entire diagonal $j-i=a-(2k+2)$ blue.

Therefore, all vertices on the diagonals
\[
j-i\geq a-2(k+1)
\]
are blue, completing the induction step. Applying the induction step through $k=a-1$, all vertices satisfying
\[
j-i\geq a-2a=-a
\]
are blue. By Definition \ref{def:UpdatedDefAD_{m,n}}, every vertex of $DR(a,b)$ satisfies this inequality. Hence, $F_3$ admits a chronological ordering starting from $B$ that forces every vertex of $DR(a,b)$ blue. The first three rounds of this process are shown in Figure \ref{fig:F3Process}.
\end{proof}

Thus, $F_1$, $F_2$, and $F_3$ are forcing processes starting from the same initial blue set $B$. We now use these processes, together with Theorem \ref{prop:twoways}, to determine the $1$-leaky forcing number of $DR(a,b)$.

\begin{figure}[h!]
    \centering

\tikzset{every picture/.style={line width=0.75pt}} %set default line width to 0.75pt        


    \caption{
The first three rounds of the forcing process $F_2$ on $DR(2,3)$, illustrating the alternating horizontal and vertical forces along consecutive diagonals of the form $i+j=c$, shown from left to right.}
    \label{fig:F2Process}
\end{figure}

\begin{figure}[h!]
    \centering

\tikzset{every picture/.style={line width=0.75pt}} %set default line width to 0.75pt        

%
    \caption{
The first three rounds of the forcing process $F_3$  on $DR(2,3)$, illustrating the alternating horizontal and vertical forces along consecutive diagonals of the form $j-i=c$, shown from left to right.}
    \label{fig:F3Process}
\end{figure}

\begin{thm}\label{thm:AR-1-leaky}
For all $a,b \in \mathbb{N}$, the $1$-leaky forcing number of the Diamond rectangle graph $DR(a,b)$ is $a+b$.
\end{thm}

\begin{proof}
Consider the Diamond rectangle graph $DR(a,b)$ with $a,b\in\mathbb{N}$, and recall the initial set
\[
B=\{(i,j)\in V(DR(a,b)):i+j=a-1\}\cup\{(i,j)\in V(DR(a,b)):j=i+a\}
\]
from Lemma \ref{lem:ZF Upper Bound}. By Lemma \ref{lem:ZF Upper Bound}, $|B|=a+b$, and $F_1$ gives a forcing process starting from $B$. By
Proposition~\ref{prop:F2F3}, $F_2$ and $F_3$ also give forcing processes
starting from $B$. Thus, $F_1\cup F_2\cup F_3\subseteq \mathcal{F}(B)$.

By Theorem~\ref{prop:twoways}, since $B$ is a zero forcing set, to show that it is a $1$-leaky forcing set, it is enough to show that every vertex
$(i,j)\in V(DR(a,b))\setminus B$ can be forced by two distinct vertices.
As in Lemma~\ref{lem:ZF Upper Bound}, every such vertex can be forced by
its left neighbor, so
\[
(i-1,j)\to(i,j)
\]
is a force in $F_1$. We now show that each vertex outside $B$ has a second distinct
force coming from either $F_2$ or $F_3$ by showing that each such vertex can also be
forced from either below or above.

By Definition \ref{def:UpdatedDefAD_{m,n}}, every vertex $(i,j)\in V(DR(a,b))$ satisfies $-a\leq j-i\leq a$. Since $(i,j)\notin B$, we cannot have $j-i=a$. Therefore, we have
\[
-a\leq j-i\leq a-1.
\]
Thus, $j-i$ takes one of the $2a$ integer values from $-a$ to $a-1$.
The values in this range can be divided into the two families
\[
\begin{aligned}
(1)\quad & j=i+a-(2k+1), && 0\leq k\leq a-1,\\
(2)\quad & j=i+a-(2k+2), && 0\leq k\leq a-1,
\end{aligned}
\]
where the first family corresponds to the values
\[
a-1,a-3,\ldots,-a+1
\]
of $j-i$, while the second corresponds to
\[
a-2,a-4,\ldots,-a.
\]
Together, these are exactly the integers from $-a$ to $a-1$, and the two families are disjoint. Therefore, every vertex outside $B$ lies on exactly one of these diagonals. We will show that vertices in the first family have a second force in $F_2$, while vertices in the second family have a second force in $F_3$.

First, suppose $j=i+a-(2k+1)$ for some $0\leq k\leq a-1$.
Then
\[
i+j=a-1+2(i-k).
\]
Since $(i,j)\notin B$, we know that $i+j\neq a-1$.
Together with the bounds
\[
a-1\leq i+j\leq 2b+a-1,
\]
this implies that $i+j=a+1+2r$ for some $0\leq r\leq b-1$. Hence,
\[
i+(j-1)=a+2r,
\]
so $(i,j-1)\to(i,j)$ is one of the vertical forces
in $F_2$. Since $(i-1,j)\to(i,j)$ is the force into $(i,j)$ from $F_1$ and $(i,j-1)\neq(i-1,j)$, the force $(i,j-1)\to(i,j)$ gives a second distinct force into $(i,j)$ satisfying $j=i+a-(2k+1)$ for $0\leq k\leq a-1$.

Next, suppose $j=i+a-(2k+2)$ for some $0\leq k\leq a-1$. Then
\[
j+1=i+a-(2k+1).
\]
Moreover,
\[
i+j=a+2(i-k-1),
\]
so the defining bounds of $DR(a,b)$ give $i+j=a+2r$ for some $0\leq r\leq b-1$. Thus, $(i,j+1)$ belongs to $DR(a,b)$ and lies on the diagonal
\[
j-i=a-(2k+1).
\]
Therefore, $(i,j+1)\to(i,j)$ is one of the vertical forces in $F_3$. Since $(i-1,j)\to(i,j)$ is the force into $(i,j)$ from $F_1$ and $(i,j+1)\neq(i-1,j)$, the force $(i,j+1)\to(i,j)$ gives a second distinct force into every vertex $(i,j)$ satisfying
$j=i+a-(2k+2)$ for $0\leq k\leq a-1$.

Thus, every vertex in $V(DR(a,b))\setminus B$ can be
forced by two distinct vertices: once by a force in $F_1$,
and once by a force in either $F_2$ or $F_3$.
Therefore, by Theorem~\ref{prop:twoways}, $B$ is a
$1$-leaky forcing set. Since $|B|=a+b$,
\[
Z_{(1)}(DR(a,b))\leq a+b.
\]
Moreover, by Theorem~\ref{thm: ZF Number of HD(m,n)}, $Z(DR(a,b))=a+b$.
Thus, Lemma~\ref{lem:moreleaksinq} gives
\[
a+b=Z(DR(a,b))\leq Z_{(1)}(DR(a,b))\leq a+b,
\]
and hence $Z_{(1)}(DR(a,b))=a+b.$
\end{proof}

Having established the $1$-leaky forcing number of $DR(a,b)$, we now turn to determining $Z_{(\ell)}(DR(a,b))$ for all $\ell \ge 2$. First, we have the following lemma which counts the number of boundary vertices of $DR(a,b)$. 

\begin{lem}\label{lem:degtwoverts}
For any $a,b\in\mathbb{N}$, the number of degree $2$ vertices of $DR(a,b)$ is $2(a+b)$, i.e.
\[
|\{v\in V(DR(a,b)):\deg_{DR(a,b)}(v)=2\}|=2(a+b).
\]
\end{lem}

\begin{proof}
By Theorem \ref{thm:DR-visual}, the vertices of degree $2$ in $DR(a,b)$
are exactly the vertices lying on one of the four boundary lines
\[
i+j=a-1,\qquad
i+j=2b+a-1,\qquad
i-j=a,\qquad
j-i=a.
\]
Moreover, these four boundary conditions are disjoint.

Similar to the counting argument in Lemma \ref{lem:ZF Upper Bound}, we have
\[
|\{(i,j)\in V(DR(a,b)):i+j=a-1\}|=a,
\qquad
|\{(i,j)\in V(DR(a,b)):i+j=2b+a-1\}|=a,
\]
and
\[
|\{(i,j)\in V(DR(a,b)):i-j=a\}|=b,
\qquad
|\{(i,j)\in V(DR(a,b)):j-i=a\}|=b.
\]
Since these four sets are disjoint,
\[
\bigl|\{v\in V(DR(a,b)):\deg_{DR(a,b)}(v)=2\}\bigr|
=a+a+b+b
=2(a+b).
\]
\end{proof}

Now, applying our results from Section \ref{sec:FortsInducedSubgraphs}, we are able to obtain the following classification of the $\ell$-leaky numbers for this graph family. In particular, applying Corollary \ref{cor:LeakyForcingBoundaryEquality} and Lemma \ref{lem:degtwoverts} with $d= 2$, we have the following. 

\begin{thm}\label{cor:recover-hopi-from-versionB}
For all $a,b\in\mathbb{N}$, the Diamond rectangle graph $DR(a,b)$ satisfies
\[
Z_{(2)}\!\bigl(DR(a,b)\bigr)=Z_{(3)}\!\bigl(DR(a,b)\bigr)=2(a+b),
\]
and for all $\ell\ge 4$,
\[
Z_{(\ell)}\!\bigl(DR(a,b)\bigr)=a+b+2a b.
\]
\end{thm}

% \begin{proof}
% Let $n_1,n_2\in\mathbb{N}$ and $\ell\ge 2$. Considering $DR(a,b)$, which is an induced subgraph of $P_{a+b}\square P_{a+b}$, by Corollary \ref{cor:LeakyForcingBoundaryEquality}, we have \[ |S_\ell| \le Z_{(\ell)}(DR(a,b)) \le |\delta DR(a,b)|,\] where $S_\ell = \{v\in V(DR(a,b)) : \deg_{(DR(a,b)}(v) \le \ell\}$ and \[\delta DR(a,b) = \{v\in V(DR(a,b)) : \deg_{DR(a,b)}(v) \le 3\}.\]
% Since every vertex in $DR(a,b)$ satisfies $\deg(v)\in\{2,4\}$, the boundary set $\delta DR(a,b)$ simplifies to \[\delta DR(a,b) = \{v\in V(DR(a,b)) : \deg_{DR(a,b)}(v) = 2\}.\] 

% Now, again because of the degrees of vertices in $DR(a,b)$, when $\ell = 2 \ \text{or} \ 3$, we have \[S_\ell = \{v \in V(DR(a,b)) : \deg_{DR(a,b)}(v) = 2\}\]
% which is the same as the boundary set, and thus we get \[Z_{(\ell)}(DR(a,b)) = |S_\ell |.\]
% As in Lemma \ref{lem:ZF Upper Bound}, one can show that $|S_\ell|=2(m+n)$, and hence we have
% \[
% Z_{(\ell)}\!\bigl(HD(m,n)\bigr)=2(m+n) \ \text{ for } \ell=2,3.
% \]

% When $\ell \ge 4$, since $\Delta(DR(a,b)) = 4$, we have $Z_{(\ell)}(HD(m,n)) = |V(HD(m,n))|$ by Lemma \ref{lem:lowdegree}. Thus, applying Proposition \ref{prop: Order of HD(m,n)} we have $Z_{(\ell)}(HD(m,n)) = m+n+2mn$ for all $\ell \ge 4$. 
% \end{proof}

\begin{proof}
Let $a,b\in\mathbb{N}$. First, suppose $\ell=2$ or $\ell=3$, and set $H:=DR(a,b)$. Since $H$ is an induced subgraph of $P_{a+b}\square P_{a+b}$, Corollary \ref{cor:LeakyForcingBoundaryEquality} gives
\begin{equation}\label{eq:inProofBoundingLeaky}
|S_\ell|\leq Z_{(\ell)}(H)\leq |\delta H|,
\end{equation}
where recall $S_\ell=\{v\in V(H):\deg_H(v)\leq \ell\}$ and $\delta H=\{v\in V(H):\deg_H(v)\leq 3\}$. By Theorem \ref{thm:DR-visual}, every vertex of $H$ has degree either $2$ or $4$, and so we get $S_\ell=\{v\in V(H):\deg_H(v)=2\}=\delta H$. Thus, the bounds in \eqref{eq:inProofBoundingLeaky} give $Z_{(\ell)}(H)=|S_\ell|$. By Lemma \ref{lem:degtwoverts}, $|S_\ell|=2(a+b)$,
and hence
\[
Z_{(2)}(H)=Z_{(3)}(H)=2(a+b).
\]

Now suppose $\ell\geq 4$. Since every vertex of $H$ has degree at most $4$, Lemma \ref{lem:lowdegree} gives $Z_{(\ell)}(H)=|V(H)|$.
By Theorem \ref{thm:DR-visual}, $|V(H)|=a+b+2ab$, and therefore,
\[
Z_{(\ell)}(H)=a+b+2ab
\qquad\text{for all }\ell\geq 4.
\]
\end{proof}

We summarize the values of the zero forcing and $\ell$-leaky forcing numbers established for $DR(a,b)$ in Table~\ref{tab:DR-leaky-forcing}.

\begin{table}[h!]
\centering
\begin{tabular}{c|c}
\hline
\multicolumn{1}{|c|}{Number of leaks} &
\multicolumn{1}{c|}{Forcing number of $DR(a,b)$} \\
\hline
$0$ & $a+b$ \\
$1$ & $a+b$ \\
$2\leq \ell\leq 3$ & $2(a+b)$ \\
$\ell\geq 4$ & $a+b+2ab$ \\
\hline
\end{tabular}
\caption{Zero forcing and $\ell$-leaky forcing numbers of $DR(a,b)$.}
\label{tab:DR-leaky-forcing}
\end{table}

\section{Conclusion and Future Work}\label{sec:conclusion}\label{sec:FutureWork}

The results in this paper began with the study of Diamond rectangle graphs and led to more general tools for determining leaky forcing numbers on subgraphs of grid graphs. Although our main application was to Diamond rectangle graphs, these results suggest several directions for future work.

One direction is to consider subgraphs of higher-dimensional grid graphs $P_{n_1}\square\cdots\square P_{n_d}$.
For example, in three dimensions, one can consider prisms of the form $H\square P_m$, where $H$ is a Diamond rectangle graph. Unlike the $2$-dimensional case, these graphs can have several different boundary degrees, but some of the ideas and results developed in this paper can still be applied. In particular, the cases $\ell=0$ and $\ell=1$ appear to require additional work. Studying these examples would provide a first step toward extending the results of this paper to higher-dimensional grid graphs.

\qquad Another direction is to consider other planar families where our general results apply. One natural class is given by \emph{polyomino graphs}, which arise from the combinatorial objects known as \emph{polyominoes}. More precisely, if $S\subseteq V(P_{n_1}\square P_{n_2})$ is nonempty and  $(P_{n_1}\square P_{n_2})[S]$ is connected, then the induced subgraph $(P_{n_1}\square P_{n_2})[S]$ gives the adjacency graph of a polyomino region. This includes many grid-based families beyond Diamond rectangle graphs, such as Young-diagram shapes, convex polyomino regions, and grid graphs obtained by removing rectangular holes. These families provide additional settings in which to study how the structure of the boundary affects the leaky forcing number.

\bibliography{ref}

\end{document}